\documentclass[12pt, reqno]{amsart}

\usepackage{amsmath,amsfonts,amsthm,amssymb,color}
\usepackage{amsmath,bm}
\usepackage[utf8]{inputenc}

\usepackage{amsfonts}
\usepackage{amsmath}
\usepackage{amssymb}
\usepackage{amsthm} 
\usepackage{latexsym}
\usepackage{bbm}
\usepackage{mathrsfs}
\usepackage{nameref}
\usepackage{url}

\usepackage{xcolor}

\usepackage{float}
\usepackage[caption = false]{subfig}
\usepackage[final]{graphicx}
\usepackage{systeme}
\usepackage{cancel}
\usepackage{blkarray}
\usepackage{placeins}
\usepackage{mathrsfs}

\usepackage[T1]{fontenc}

\makeatletter
     \def\section{\@startsection{section}{1}%
     \z@{.7\linespacing\@plus\linespacing}{.5\linespacing}%
     {\bfseries
     \centering
     }}
     \def\@secnumfont{\bfseries}
     \makeatother

\usepackage[colorlinks=true,urlcolor=blue]{hyperref}

\newcommand{\R}{\mathbb R}

\newtheorem{theorem}{Theorem}[section]
\newtheorem{lemma}[theorem]{Lemma}
\newtheorem{proposition}[theorem]{Proposition}
\newtheorem{corollary}[theorem]{Corollary}
\theoremstyle{definition}
\newtheorem{definition}[theorem]{Definition}

\theoremstyle{remark}
\newtheorem{remark}{Remark}
\numberwithin{equation}{section}
\DeclareMathOperator{\Var}{Var}

\newcommand\inner[2]{\langle #1, #2 \rangle}
\DeclareMathOperator*{\w}{w-\!\!} 

\newcommand\bel[1]{\begin{equation}\label{#1}}
\newcommand\ee{\end{equation}}

\newcommand\bE{{\mathbb{E}}}

\newcommand{\norm}[1]{\left\lVert#1\right\rVert}

 \DeclareGraphicsExtensions{.pdf,.png,.jpg,.tiff}

\restylefloat{figure}

\begin{document}
\title[MLE for a stochastic  wave equation with Malliavin calculus]{Statistical
inference for a stochastic  wave equation with Malliavin calculus.}

\author[F. Delgado-Vences]{Francisco Delgado-Vences} 
\address{{\it F. Delgado-Vences}. Conacyt  Research Fellow - Universidad Nacional Aut\'onoma de M\'exico. Instituto de Matem\'aticas, Oaxaca, M\'exico}
\email{delgado@im.unam.mx}

\author[J.J. Pavon-Espa\~{n}ol]{Jose Julian Pavon-Espa\~{n}ol} 
\address{  {\it J.J. Pavon-Espa\~{n}ol}.
		Facultad de ciencias - Universidad Nacional Aut\'onoma de M\'exico.  
		Ciudad Universitaria, Ciudad de M\'exico}
	\email{julian.pavon2@ ciencias.unam.mx }

\begin{abstract}

In this paper we study asymptotic properties of the maximum likelihood estimator (MLE) for the speed of a stochastic  wave equation. We follow a well-known spectral approach to write the solution as a Fourier series, then we project the solution to a $N$-finite dimensional space and find the estimator as a function of the time and  $N$.  We then show consistency of the MLE using classical stochastic analysis. Afterward we prove the asymptotic normality using the Malliavin-Stein method. We also study asymptotic properties of a discretized version of the MLE for the parameter. We provide this  asymptotic analysis of the proposed estimator as the number of Fourier modes, $N$, used in the estimation and the observation time go to infinity. Finally, we illustrate the theoretical results with some numerical experiments.

\end{abstract}

\maketitle

\medskip\noindent

\allowdisplaybreaks

\section{Introduction and main results}

We consider the stochastic partial differential  equation (SPDE)
\begin{equation}\label{Ec1}
\frac{\partial^2 u}{\partial t^2}=\lambda\,\frac{\partial^2 u}{\partial x^2}
 + \sigma \dot{W}(t),\ 0<t<T,\ 0<x<\pi,
\end{equation}
where $W$ is a cylindrical Brownian motion over $L_2((0,\pi))$.
We will assume
\begin{align}
 \label{coef}&\sqrt{\lambda} >0,\  \ \sigma> 0;\\
\label{fron}
&u|_{t=0}=\frac{\partial u}{\partial t}\Bigg|_{t=0}=0,\ \
u|_{x=0}=u|_{x=\pi}=0.
\end{align}

In this paper we study asymptotic properties of the MLE for $\lambda$. The theory of SPDEs has become highly important given the number of applications to other areas of science. Considerable progresses have been established on the general theory for SPDEs, we refer, for instance, to the books \cite{DPZ}, \cite{GM}, \cite{LoR} or \cite{Chow}. Regarding statistical inference for SPDEs, we see that this is a relatively new research area. The well-known spectral approach, that is used in this paper, consists of writing the solution of \eqref{Ec1} $u$ as a Fourier series, and then using the first $N$ modes of the solution in the estimation on a fixed window time $[0,T]$. Liu and Lototsky \cite{LiuL} have studied \eqref{Ec1} in presence of damping, and Cialenco, Delgado and Kim \cite{CDK} have studied the MLE for the heat equation not only when $N\rightarrow\infty$ but also when $N,T\rightarrow\infty$. The objective of this paper is extend the results from \cite{LiuL} in a similar way to \cite{CDK}, note that in our case the wave equation is a hyperbolic equation in contrast to the heat equation, it is a parabolic one; for which there are several more detailed literature, see for example \cite{Cia} or \cite[Chapter 6]{LoR} for an overview of the subject.  Liu and Lototsky \cite{LiuL2} have studied the general case for hyperbolic equations with two parameters when $N\rightarrow\infty$.  For Hyperbolic SPDEs we cite \cite{Ja1}, where the author introduce a different method for study statistical inference for SPDEs, meaning the so-called minimum contrast estimators, that is applied for a stochastic wave equations, he prove strong consistency and asymptotic normality for this estimator.

We now present the main result of this work. 

\begin{theorem}\label{main-th2}
Under assumptions \eqref{fron} and \eqref{coef} , the MLE $\widehat{\lambda}_{N,T}$ is strongly consistent as $N,T\rightarrow \infty$. Moreover, the next limit holds
\begin{equation*}
\begin{split}
&\lim_{N,T\to \infty} T N^{3/2}(\widehat{\lambda}_{N,T}-\lambda)=
\mathcal{N}\left(0, 12\lambda\right),
\end{split}
\end{equation*}
in distribution.
\end{theorem}

In \cite{LiuL}, the authors have studied the damped version of \eqref{Ec1}. 
\begin{equation}\label{Ec2}
\frac{\partial^2 u}{\partial t^2}=\lambda_1\,\frac{\partial^2 u}{\partial x^2}
+\lambda_2\,\frac{\partial u}{\partial t} + \sigma \dot{W}(t),\ 0<t<T,\ 0<x<\pi,
\end{equation}
where the same boundary conditions are imposed. But different conditions about the coefficients are  imposed:
\begin{align}
    \label{coef1}\sqrt{\lambda_1} \geq 1,\ \  |\lambda_2|\leq 1;\ \ \sigma> 0.
\end{align}

 We enunciate their result in \cite{LiuL}, observe that asymptotic properties of the estimator are studied  only as the number of the Fourier coefficients increases . 
\begin{theorem}\label{main-th1}
Under assumptions \eqref{fron} and \eqref{coef1}, both estimators are strongly consistent as $N\rightarrow\infty$. Moreover, the next limits hold
\begin{equation*}
\begin{split}
&\lim_{N\to \infty} N^{3/2}(\widehat{\lambda}_{1,N}-\lambda_1)=
\mathcal{N}\left(0, \frac{3\lambda_1}{A(\lambda_2,T)}\right), \\
&\lim_{N\to \infty} N^{1/2}(\widehat{\lambda}_{2,N}-\lambda_2)=
\mathcal{N}\left(0, \frac{1}{A(\lambda_2,T)}\right)
\end{split}
\end{equation*}
in distribution. Where 
\begin{equation}
\label{coef-th2}
A(\lambda_2,T)=
\begin{cases}
\displaystyle \frac{e^{\lambda_2T}-\lambda_2T-1}{2\lambda_2^2},&{\ \rm if}\ \lambda_2\not=0;\\
\displaystyle  \frac{T^2}{4},& {\ \rm if}\ \lambda_2=0,
\end{cases}
\end{equation}
for $T>0$ and $\lambda_2\in \R$; note that $A(\lambda_2,T)>0$ for all $T>0$ and $\lambda_2\in \R$. In particular without presence of damping, the next limits hold
$$
\lim_{N\to \infty}\widehat{\lambda}_{N}=\lambda
$$
with probability one and
\begin{equation*}
\begin{split}
&\lim_{N\to \infty} N^{3/2}(\widehat{\lambda}_{N}-\lambda)=
\mathcal{N}\left(0, \frac{12\lambda}{ T^2}\right)
\end{split}
\end{equation*}
in distribution.
\end{theorem}

Observe that in the Theorem \ref{main-th1} we have only the limit when $N$ goes to infinity, while in our main result, Theorem \ref{main-th2}, the result is proved when both, $N$ and $T$, go to infinity. We prove the asymptotic normality, of our main result, with the use of the so-called Malliavin-Stein's approach.

The paper is organized as follows. In section \ref{sect-framework}, we introduce the basic notions of the  solution to \eqref{Ec1}. Then in section \ref{sect-MLE} we prove the main result using classical stochastic analysis and in the subsection \ref{sect-MSap} we prove the asymptotic normality using the Malliavin-Stein's approach. In section \ref{sect-Discrete} we introduce the discretized version of the MLE  using the first  measured $N$ Fourier modes of the solution at $M$ fixed time grid points uniformly spaced over the time interval $[0,T]$. We study a weak asymptotic properties; first we prove the discretized MLE $\widehat{\lambda}_{N,M}$ is weakly consistent as $N,M\rightarrow \infty$ and finally putting  assumptions  over $N,M$ we prove the asymptotic normality of $\widehat{\lambda}_{N,M}$ with the same rate  of convergence seen on \ref{main-th1}. And to conclude we illustrate the theoretical results with some numerical experiments.

\section{Framework}
\label{sect-framework}

Let $(\Omega, \mathscr{F}, \{\mathscr{F}_t\}_{t\geq 0},\mathbb{P})$ be  a filtered  probability space, where the filtration $\{\mathscr{F}_t\}_{t\geq 0}$ satisfies the usual conditions. We consider a cylindrical Brownian motion over $L_2((0,\pi))$ as follows. 
A cylindrical Brownian
motion $W=W(t)$, $t\geq 0$, over the Hilbert space $H=L_2((0,\pi))$ is a linear mapping
$$
W: f\mapsto W_f(\cdot)
$$
from
$H$ to the space of zero-mean Gaussian processes such that, for every
$f,g\in H$ and $t,s>0$, $\{W_f (t)\}_{t\geq0 }$ is a one-dimensional Brownian motion  and
\begin{equation}
\label{CondicionBrown}
\bE\big(W_f(t)W_g(s)\big)=\min(t,s)\inner{f}{g}_H.
\end{equation}
Note that a cylindrical Brownian motion can be represented as a $\mathbb{P}$-a.s. convergent series
$$W_f (t) = \sum_{k\geq 1} \inner{f}{e_k}_H W_{e_k}(t)$$
where $\{e_k,\ k\geq 1\}$ is an orthonormal basis in $H$, and $\{W_{e_k}(t)\}_{k=1}^\infty$ is a collection of real independent standard  Brownian motions since the $\{e_k\}_{k\geq 1}$ are  orthonormal.

The equation \eqref{Ec1} can be interpreted  as a system of two first-order
It\^{o}'s equations
\begin{equation}
du=vdt,\ dv=\lambda u_{xx}dt+\sigma dW(t).
\end{equation}

For $\gamma\in \R$, define the Hilbert space $H^{\gamma}$ as the closure
of the set of smooth compactly supported functions on $(0,\pi)$
with respect to the norm
\begin{equation}
\label{SobSp}
\|f\|_{\gamma}=\left(\sum_{k\geq1} k^{2\gamma}f_k^2\right)^{1/2},
\end{equation}
where $f_k=\sqrt{\frac{2}{\pi}}\int_0^{\pi} f(x)\sin(kx)dx$. Note that  each of the functions $\sin(kx)$ belongs to every $H^{\gamma}$, and
if $f$ is  twice continuously-differentiable on $(0,\pi)$ with $f(0)=f(\pi)=0$, then  $f\in H^{1}$. More generally, every  $f\in H^{\gamma}$ can be identified with a sequence $\{f_k,\ k\geq 1\}$ of real
numbers such that $\sum_{k\geq 1}k^{2\gamma}f_k^2<\infty$.

Given $\gamma> 0$, $f\in H^{-\gamma}$ and
$g\in H^{\gamma}$, we define 
$$
\inner{f}{g}=\sum_{k\geq 1} f_kg_k;
$$
if $f,g\in L_2((0,\pi))$, then
$$
\inner{f}{g}=\int_0^{\pi} f(x)g(x)dx.
$$
In other words, $\inner{\cdot}{ \cdot}$ is the duality between $H^{\gamma}$ and $H^{-\gamma}$
relative to the inner product in $H^0=L_2((0,\pi))$. Be aware that 
$$H^{\gamma} \subset L_2 ((0,\pi))\subset H^{-\gamma},$$
where all are embeddings.

First we write the definition of a  solution to \eqref{Ec1}.

\begin{definition}
\label{def-sol}
An adapted process $u\in L_2\big(\Omega\times(0,T)\times (0,\pi)\big)$
is called  solution of \eqref{Ec1} if there
exists an adapted process $v$ such that
\begin{enumerate}
\item $v\in L_2\big(\Omega; L_2((0,T);H^{-1})\big)$;
\item For every twice continuously-differentiable on $(0,\pi)$ function $f=f(x)$
with $f(0)=f(\pi)=0$, the equalities
\begin{equation}
\label{SWE3}
\begin{split}
\inner{u(t,\cdot)}{f}&=\int_0^t \inner{v(t,\cdot)}{f}(s)ds,\\
\inner{v(t,\cdot)}{f}&=\int_0^t \lambda \inner{u(t,\cdot)}{f''}ds+W_f(t)
\end{split}
\end{equation}
hold for all $t\in [0,T]$ on the same set of probability one.
\end{enumerate}
\end{definition}

A theorem on the existence of the solution to the wave equation can be found in several books, see for instance \cite[Theorem 6.8.4]{Chow}, \cite[Theorem 4.3.3]{LoR}. However, for our purpose it is better to know the explicit  solutions to the system of the equations in order to study the properties of the estimators. Therefore, we enunciate the  explicit solutions for the case without damping founded on \cite{LiuL}.

\begin{theorem}\cite[Theorem 2.1]{LiuL}
\label{exis}
Under assumptions \eqref{coef} and \eqref{fron}, equation \eqref{Ec1}
has a unique solution and, for every $\gamma<1/2$,
\begin{equation}
\label{espacios}
u \in {L}_2\big(\Omega;L_2((0,T);{H}^{\gamma})\big);\ \
v\in  {L}_2\big(\Omega; L_2((0,T);{H}^{\gamma-1})\big).
\end{equation}
Moreover, $u$ and $v$ have the next representation 
\begin{equation}
\label{SOL}
u(t,x)=\sqrt{\frac{2}{\pi}}\sum_{k\geq 1} u_k(t)\sin(kx),\ \
v(t,x)= \sqrt{\frac{2}{\pi}}\sum_{k\geq 1} v_k(t)\sin(kx),
\end{equation}
where
\begin{equation}
\label{Osc3}
\begin{split}
u_k(t)&=\frac{\sigma }{\ell_k}\int_0^t\sin\big(\ell_k(t-s)\big)dw_k(s),\\
v_k(t)&=\frac{\sigma }{\ell_k}\int_0^t
\ell_k\cos\big(\ell_k(t-s)\big) dw_k(s),
\end{split}
\end{equation}
where $\ell_k=\sqrt{\lambda}k$.
\end{theorem}

\begin{remark}
Observe that  $\{u_k(t)\}_{ k\ge 1}$ constitutes a sequence of independent Gaussian process as well as $\{v_k(t)\}_{ k\ge 1}$, since $\{w_k(t)\}_{ k\ge 1}$ is indeed. This will be important for some calculations to obtain the main result of this paper.
\end{remark}

\section{Maximum Likelihood Estimators}
\label{sect-MLE}
In this section we will investigate the MLE for the  parameter $\lambda$, in particular to  study its asymptotic properties.

In \eqref{SWE3} define $f_k (x)=\sqrt{2/\pi} \sin(kx)$ and note that
$u_k(t)=\inner{u(t,\cdot)}{f_k}$, $v_k(t)=\inner{v(t,\cdot)}{f_k}$, $w_k=W_f$. Then
\begin{equation}
\label{Osc1}
u_k(t)=\int_0^tv_k(s)ds,\ v_k(t)=-\lambda k^2\int_0^tu_k(s)ds+\sigma w_k(t),
\end{equation}
or
\begin{equation}
\label{Osc2}
\ddot{u}_k(t)+\lambda_2 \dot{u}_k(t)+\lambda k^2u_k(t)=\sigma \dot{w}_k(t),
 \ \ u_k(0)=\dot{u}_k(0)=0.
\end{equation}


By \eqref{Osc1}, we have
\begin{equation}
\label{system}
u_k(t)=\int_0^tv_k(s)ds,\ v_k(t)=-\lambda k^2\int_0^tu_k(s)ds+\sigma w_k(t).
\end{equation}

For each $k\geq 1$, the processes $u_k$, $v_k$, and $w_k$  generate measures
$\mathbf{P}^u_k$, $\mathbf{P}^v_k$, $\mathbf{P}^w_k$ in the
space $\mathcal{C}((0,T);\R)$
of continuous, real-valued functions on $[0,T]$.
Since $u_k$ is a continuously-differentiable function, the
measures $\mathbf{P}^u_k$ and
$\mathbf{P}^w_k$ are mutually singular. On the other hand, we can write
\begin{equation}
\label{v-e}
dv_k(t)=F_k(v)dt+\sigma dw_k,
\end{equation}
where
$F_k(v)=-\lambda_1k^2\int_0^tv_k(s)ds$ is a
non-anticipating functional of $v$.
Thus, the process $v$ is a process of diffusion type. Further analysis shows
that the measure $\mathbf{P}^v_k$ is
absolutely continuous with respect to the measure $\mathbf{P}^w_k$, and using  \cite[Theorem 7.6]{LSh1}
\begin{equation}
\label{density1-e}
\begin{split}
\frac{d\mathbf{P}^v_k}{d\mathbf{P}^w_k}(v_k)&=\exp\Bigg(\frac{1}{\sigma^2}\int_0^T\big(-\lambda k^2u_k(t)+\lambda_2 v_k(t)\big)dv_k(t)\\
&\hspace{.4cm} -\frac{1}{2\sigma^2}
\int_0^T\big(-\lambda k^2u_k(t)+\lambda_2v_k(t)\big)^2dt\Bigg).
\end{split}
\end{equation}
Since the processes $w_k$ are independent for different $k$, so are the processes $v_k$.
Therefore, the measure $\mathbf{P}^{v,N}$ generated in $\mathcal{C}((0,T);\R^N)$
by the vector process
$\{v_k,\ k=1,\ldots,N\}$ is absolutely continuous with respect to the
 measure $\mathbf{P}^{w,N}$ generated in $\mathcal{C}((0,T);\R^N)$
by the vector process
$\{w_k,\ k=1,\ldots,N\}$, and the density is
\begin{equation}
\label{density2-e}
\begin{split}
\frac{d\mathbf{P}^{v,N}}{d\mathbf{P}^{w,N}}(v_k)&=\exp\Bigg(\frac{1}{\sigma^2}\sum_{k=1}^N\int_0^T-\lambda k^2u_k(t)dv_k(t)\\
&-\frac{1}{2\sigma^2}
\sum_{k=1}^N\int_0^T\big(-\lambda k^2u_k(t)\big)^2dt\Bigg);
\end{split}
\end{equation}
the corresponding log-likelihood ratio is
\begin{equation}
\label{log-lik}
\begin{split}
Z_{N,T}(\lambda)&=\frac{1}{\sigma^2}\sum_{k=1}^N\Bigg(\int_0^T-\lambda k^2u_k(t)dv_k(t)-\frac{1}{2\sigma^2}
\int_0^T\lambda^2 k^4u_k^2dt\Bigg).
\end{split}
\end{equation}
Introduce the following notations:
\begin{equation}
\label{notations}
\begin{split}
&J_{N,T}=\sum_{k=1}^N k^4\int_0^T u_k^2(t)dt,\\  
&B_{N,T}=-\sum_{k=1}^N k^2 \int_0^Tu_k(t)dv_k(t),\ \xi_{N,T}=\sum_{k=1}^N k^2 \int_0^Tu_k(t)dw_k(t).
\end{split}
\end{equation}
Be aware that the numbers $J$ and $B$ are computable from the observations of $u_k$ and $v_k$, $k=1,
\ldots, N$, and also   we have
\begin{align}
\label{ident-B}  B_{N,T}=\lambda J_{N,T}-\sigma\xi_{N,T}.
\end{align}
We consider the problem of estimating
$\lambda$  from the observations
$
\{u_k(t),\ v_k(t) : k=1,\ldots, N,\ t\in [0,T]\}.
$ Then the corresponding log-likelihood ratio is
\begin{equation}
\label{log-lik0}
\begin{split}
Z_{N,T}(\lambda)&=\frac{1}{\sigma^2}\Big(\lambda B_{N,T}-\frac{\lambda^2}{2} J_{N,T}\Big).
\end{split}
\end{equation}
From this expression, we get the estimator
\begin{equation}\label{estimator0}
    \widehat{\lambda}_{N,T}=\frac{B_{N,T}}{J_{N,T}}.
\end{equation}
Moreover, we can compute the Fisher information related to $\frac{d\mathbf{P}^{v,N}}{d\mathbf{P}^{w,N}}(v_k)$.
For simplicity, set $u(0)=0$. Namely,
\begin{align*}
  \mathcal{I}_{N,T} & := \int \left|  \frac{\partial}{\partial \lambda} \log \frac{d\mathbf{P}^{v,N}}{d\mathbf{P}^{w,N}} \right|^2 d\mathbf{P}^{w,N} \\
  & = - \int  \frac{\partial^2}{\partial \lambda^2} \log \frac{d\mathbf{P}^{v,N}}{d\mathbf{P}^{w,N}} d\mathbf{P}^{w,N}  \\
  & = \frac{1}{\sigma^2} \bE[J_{N,T}]=\frac{1}{\sigma^2}\sum_{k=1}^N k^4 \int_0^T\bE [u_k^2 (t)] dt\\
  &=\frac{1}{\sigma^2}\sum_{k=1}^N k^4 \frac{\sigma^2}{\ell_k^2}\left(\frac{T^2}{4}+\frac{\cos(2\ell_k T)-1}{8\ell_k^2}\right)\\
  &\simeq N^3 \frac{T^2}{12\lambda} ,\quad \mbox{as}\ N\to \infty.
\end{align*}
In particular, be aware that $\mathcal{I}_{N,T}\to\infty$ as $N\to\infty$ when $T$ is fixed, or $N,T\rightarrow \infty$ .

 From straightforward calculations we get

\begin{align}
    \bE u_k^2(t)&=\frac{\sigma^2}{\ell_k^2}\left(\frac{t}{2}-\frac{\sin(2\ell_k t)}{4\ell_k}\right) \label{eq:2ndmomentuk}, \qquad
    \bE v_k^2(t)&=
    \frac{\sigma^2}{2 }t-\frac{\sigma^2}{2}\frac{\sin (\ell_k t)}{2\ell_k}.
\end{align}
Also, we have the next asymptotic behavior of the second moment.

\begin{lemma}\label{2momentl2}
From \ref{Osc3}, we have the following limits
\begin{align*}
    \lim_{k\rightarrow \infty}k^2 \bE u_k^2(t)&=
    \frac{\sigma^2}{\lambda}\frac{t}{2},\qquad
    \lim_{k\rightarrow \infty} \bE v_k^2(t)=\frac{\sigma^2}{2 }t.
\end{align*}
Moreover, we have
\begin{align*}
    \lim_{k\rightarrow \infty} k^2 \bE \int_0^T u_k^2(t) dt=\frac{\sigma^2 T^2}{4\lambda}, \qquad
    \lim_{k\rightarrow \infty}  \bE \int_0^T v_k^2(t)dt=\frac{\sigma^2 T^2}{4}.
\end{align*}
\end{lemma}

And using a basic result from convergent series and the last lemma we get the next corollary.
\begin{corollary}\label{lem-1}
The following limits hold
\begin{align*}
    \lim_{N\rightarrow \infty} \frac{1}{N^3}\sum_{k=1}^N k^2 \bE \int_0^T u_k^2(t) dt=\frac{\sigma^2 T^2}{12\lambda}, \qquad
    \lim_{k\rightarrow \infty} \frac{1}{N}\sum_{k=1}^N  \bE \int_0^T v_k^2(t)dt=\frac{\sigma^2 T^2}{4}.
\end{align*}
\end{corollary}

We will focus on the main result in this section. To the best of our knowledge, the asymptotic properties of $\widehat{\lambda}_{N,T}$ have not studied when $N,T \rightarrow\infty$ (\textbf{both})\footnote{This regime is called long time and large space.}.

Now, we can prove the asymptotic properties  of the estimator $ \widehat{\lambda}_{N,T}$ in the new regime.
\begin{theorem}\label{th:NT1}
The estimator $\widehat{\lambda}_{N,T}$ is strongly consistent, that is,
\begin{equation}
\lim_{N,T \to\infty} \widehat{\lambda}_{N,T} = \lambda,   \ \  \textrm{with probability one}, \label{eq:ConsNT}
\end{equation}
and asymptotically normal, i.e. , 
\begin{equation}
\w\lim_{N,T\rightarrow \infty} TN^{\frac{3}{2}} \left(\widehat{\lambda}_{N,T} -\lambda\right)= \mathcal{N}\left(0,12\lambda\right). \label{eq:NTAsymNormal1}
\end{equation}
\end{theorem}
In this section we prove the consistency of the estimator, while the asymptotical normality is proven with the Malliavin-Stein method in the following section.

\begin{proof}
We note that
\begin{align}
\widehat{\lambda}_{N,T}-\lambda & =\frac{-\sigma \xi_{1,N,T}}{J_{1,N,T}}\\
&=-\frac{\sigma \sum_{k=1}^N k^2 \int_0^{T} u_k(t)d w_k(t)}
{\sum_{k=1}^N k^4 \int_0^{T} u_k^2(t)d t} \nonumber\\
&= -\frac{\sigma \sum_{k=1}^N\chi_{k,T}}{\sum_{k=1}^N \Var\left(\chi_{k,T}\right)}
\cdot \frac{\sum_{k=1}^N \Var\left(\chi_{k,T}\right)} {\sum_{k=1}^N k^4\int_0^{T} u_k^2(t)d t}, \label{eq:thetahat-theta0}
\end{align}
where
$$
\chi_{k,T}:=k^2\int_0^{T}u_k(t)d w_k(t).
$$
From \eqref{eq:2ndmomentuk}, we have
$$
\frac{1}{k^2}\Var(\chi_{k,T})=k^2 \int_0^{T} \mathbb{E} u_k^2(t) d t
=\frac{\sigma^2}{\lambda}\left(\frac{T^2}{4}+\frac{\cos(2\ell_k T)-1}{8\ell_k^2}\right)\rightarrow \frac{\sigma^2T^2}{4\lambda},\quad \mbox{as}\ k\to \infty,
$$
and thus,
\begin{align}
\sum_{k=1}^N \Var(\chi_{k,T}) &\simeq \frac{N^3\sigma^2T^2}{12\lambda},\quad \mbox{as}\ N\to \infty.\label{asympSumVarxi}
\end{align}
Using the last two lines we can deduce that
\begin{align*}
    \frac{ \Var\left(\chi_{N,T}\right)}{\left(\sum_{k=1}^N \Var\left(\chi_{k,T}\right)\right)^2}&\simeq \frac{N^2\sigma^2 T^2 }{4\lambda}\frac{144\lambda^2}{N^6 \sigma^4 T^4} \quad \mbox{as}\ N\to \infty\\
    &=\frac{38\lambda}{N^4 \sigma^2 T^2}\quad \mbox{as}\ N\to \infty.
\end{align*}
Moreover, using  Jensen's inequality, we get that
\begin{align}
\frac{1}{k^4}\Var\left(k^4\int_0^T u_k^2(t)d t\right)
&\leq \frac{1}{k^4}\mathbb{E}\left(k^4\int_0^T u_k^2(t)d t\right)^2 \leq k^4 T\int_0^T\mathbb{E}u_k^4(t)d t\nonumber\\
&=k^4 T3\int_0^T\mathbb{E}^2[u_k^2(t)]d t\nonumber\\
&=3k^4 T\int_0^T \Bigg(\frac{\sigma^2}{\ell_k^2}\left(\frac{t}{2}-\frac{\sin(2\ell_k t)}{4\ell_k}\right) \Bigg)^2 dt\nonumber\\
&=\frac{3Tk^4\sigma^4}{\ell_k^4}\Bigg(\frac{T^3}{6}+\frac{\sin(2\ell_k T)-2\ell_k T\cos(2\ell_k T)}{4\ell_k^2}\nonumber\\
&\hspace{2cm}+\frac{T}{32\ell_k^2}-\frac{\sin(4\ell_k T)}{64\ell_k^3}\Bigg)\nonumber\\
&\rightarrow \frac{T^4\sigma^4}{2\lambda^2},\label{asympVar}
\end{align}
as $k\rightarrow\infty$. Moreover, we can prove $$\lim_{N\rightarrow\infty}\frac{1}{N^{5}}\sum_{k=1}^N\Var\left(k^4\int_0^T u_k^2(t)d t\right)=\frac{T^4\sigma^4}{10\lambda^2}$$ or $\sum_{k=1}^N\Var\left(k^4\int_0^T u_k^2(t)d t\right)\sim N^5 \frac{T^4\sigma^4}{10\lambda^2}$ as $N\rightarrow \infty$.

Hence, 
\begin{align*}
\sum_{N=1}^{\infty}\frac{ \Var\left(\chi_{N,T}\right)}{\left(\sum_{k=1}^N \Var\left(\chi_{k,T}\right)\right)^2}&\simeq\sum_{N=1}^{\infty} \frac{\frac{N^2\sigma^2T}{4\lambda}}{\frac{N^6 \sigma^4 T^4}{144 \lambda^2}}= \frac{C_1}{T^2}
\sum_{N=1}^{\infty}\frac{1}{N^{4}}\leq C_3<\infty, \\
\sum_{N=1}^{\infty}
\frac{\Var\left(N^4\int_0^T u_N^2(t)d t\right)}
{\left(\sum_{k=1}^N \Var\left(\chi_{k,T}\right)\right)^2}
&\simeq \sum_{N=1}^{\infty}\frac{\frac{N^4 T^4 \sigma^4}{2\lambda^2}}{\frac{N^6 \sigma^4 T^4}{144 \lambda^2}}
\leq C_2\sum_{N=1}^{\infty}\frac{1}{N^2}<\infty,
\end{align*}

where $C_1,C_2>0$ are some constants\footnote{Notoriously, we will denote by $C$ with subindexes generic constants that may change from line to line. } independent of $T$. Using the uniform boundedness of the above series, and employing the strong law of large numbers, we deduce that for every $\varepsilon>0$ and $T\geq T_0$, there exists $N_0>0$ independent of $T$ such that for $N\geq N_0$,
$$
\left|\frac{\sigma \sum_{k=1}^N\chi_{k,T}}{\sum_{k=1}^N \Var\left(\chi_{k,T}\right)}  \right|< \varepsilon,  \quad 
\mbox{and} \quad
\left|\frac{\sum_{k=1}^N \Var\left(\chi_{k,T}\right)}
{\sum_{k=1}^N k^4\int_0^{T} u_k^2(t)d t}-1\right|
< \varepsilon
$$
with probability one. Therefore,
\begin{equation}\label{strongLLN}
\lim_{N,T \to\infty} \frac{\sigma \sum_{k=1}^N\xi_{k,T}}{\sum_{k=1}^N \Var\left(\xi_{k,T}\right)} =0 \quad \mbox{and} \quad
\lim_{N,T \to\infty}\frac{\sum_{k=1}^N \Var\left(\xi_{k,T}\right)} {\sum_{k=1}^N k^4\int_0^{T} u_k^2(t)d t} = 1
\end{equation}
with probability one. From here, and using \eqref{eq:thetahat-theta0}, the proof of \eqref{eq:ConsNT} is complete.
This completes the proof of the consistency of the estimator.

\end{proof}

\subsection[Malliavin-Stein's approach]{Asymptotic normality of the MLE by Malliavin-Stein's approach}\label{sect-MSap}

Now, we use the Malliavin-Stein's approach to  prove the asymptotic normality of the estimator. The material of Malliavin calculus we will use is summarized in the \ref{apendix:MalCal}. First we introduce some  previous lemmas.

As before, let $u_0=0$, and for convenience, in this section we will use the following notations:
\begin{align*}
F_{N,T}&:= \frac{\xi_{N,T}}{J_{N,T}} = \frac{\sum_{k=1}^N k^2\int_0^T u_k(t)dw_k(t)}{\sum_{k=1}^N k^4\int_0^T u_k^2(t)dt} ,\\
\widehat{F}_{N,T} &:=\frac{\xi_{N,T}}{\bE J_{N,T}} = \frac{\sum_{k=1}^N k^2\int_0^T u_k(t)dw_k (t)}{\sum_{k=1}^N k^4\bE \int_0^T u_k^2(t)dt}=\frac{\xi_{1,N}}{R_{N,T}^2}   ,
\end{align*}
where $R_{N,T}^2:=\bE J_{N,T}$.

Observe that it is possible to write the stochastic process $\xi_{N,T}$ as a sum of $N$ independent multiple Wiener integral of order two, this implies that $\xi_{N,T}$ belongs to the second-order Wiener chaos. This is important since allow us to use the Malliavin-Stein method (see \cite{NourdinPeccati2012} for a nice reference on the Malliavin-Stein method).

Before proving the main theorem of this section, we will present a couple of lemmas.

\begin{lemma}
The next limits hold
\begin{align*}
    \lim_{k\rightarrow \infty} k^4  \Var (u_k^2 (t)) =\frac{\sigma^4}{\lambda^2}\frac{t^2}{2},\quad \lim_{k\rightarrow \infty}   \Var (v_k^2 (t)) = \sigma^4 \frac{t^2}{2}.
\end{align*}
Moreover, we have
\begin{align*}
    \lim_{k\rightarrow \infty} \frac{1}{N^5}\sum_{k=1}^N k^8  \Var (u_k^2 (t)) =\frac{\sigma^4}{5\lambda^2}\frac{t^2}{2}\quad \text{and } \lim_{k\rightarrow \infty} \frac{1}{N}\sum_{k=1}^N   \Var (v_k^2 (t)) =\sigma^4\frac{t^2}{2}.
\end{align*}
\end{lemma}
\begin{proof}
Since $u_k$ is Gaussian, we have
\begin{align*}
    \Var ( u_k^2 (t))&=\bE[ u_k^4]-\bE^2[ u_k^2]\\
    &=3\bE^2[ u_k^2]-\bE^2[ u_k^2]=2\bE^2[ u_k^2]\\
    &=2     \Big(\frac{\sigma^2}{\ell_k^2}\left(\frac{t}{2}-\frac{\sin(2\ell_k t)}{4\ell_k}\right)\Big)^2 
\end{align*}
Thus,
\begin{align*}
   \lim_{k\rightarrow \infty} k^4  \Var ( u_k^2 (t))&=\frac{\sigma^4}{\lambda^2}\frac{t^2}{2}.
\end{align*}
And to prove the second conclusion, notice that $\lim_{N\rightarrow\infty }\frac{1}{N^5}\sum_{k=1}^N k^4 =\frac{1}{5}$. And we can conclude  that
\begin{align*}
    \lim_{k\rightarrow \infty} \frac{1}{N^5}\sum_{k=1}^N k^8  \Var (u_k^2 (t)) =\frac{\sigma^4}{5\lambda^2}\frac{t^2}{2}.
\end{align*}
Similarly for $v_k$.
\end{proof}

The proof of the following lemma is deferred to \ref{Appendix:proofs}.

\begin{lemma}\label{conv-variance}
Let $\mathcal{H}$ be the space endowed with the inner product defined in the last section.
Let $D$ be the Malliavin derivative defined in \eqref{def-MD-tj}.
Then, we have
\begin{align*}
\sqrt{\Var\left(\frac{1}{2}\norm{ R_{N,T} D\widehat{F}_{N,T}}_{\mathcal{H}}^2\right)}&\longrightarrow 0, \mbox{ as } N,T \rightarrow \infty.
\end{align*}
\end{lemma}

Now we will prove the asymptotic normality of $R_{N,T} F_{N,T}$.

\begin{theorem}
The next limit holds in distribution
 \begin{align*}
    R_{N,T} F_{N,T}&\longrightarrow \mathcal{N}(0,1),
\end{align*}
 as $N,T\to \infty$.
\end{theorem}
\begin{proof}
Notice that $\bE\left(R^2_{N,T}\widehat{F}_{N,T}^2\right)=1.$
We split $R_{N}F_{N}$ into
\begin{equation}\label{eq:M1}
R_{N,T}F_{N,T}=R_{N,T}\left(F_{N,T}-\widehat{F}_{N,T}\right)+R_{N,T}\widehat{F}_{N,T}.
\end{equation}
We note that
$$
 R_{N,T}(F_{N,T}-\widehat{F}_{N,T}) = \frac{R_{N,T}^2}{J_{1,N}} \frac{\xi_{1,N}}{R_{N,T}}
\left(1-\frac{J_{1,N}}{R_{N,T}^2} \right).
$$
From \ref{lem-1} and \eqref{strongLLN},
$$
\frac{R_{N}^2}{J_{1,N}} \longrightarrow 1, \quad
1-\frac{J_{1,N}}{R_{N}^2} \longrightarrow 0, 
$$
 as $N,T\rightarrow\infty$, with probability 1. On the other hand, by Lemma~\ref{conv-variance}, we have
$$
\frac{\xi_{1,N}}{R_{N,T}}=R_{N,T}\widehat{F}_{N,T}\overset{d}{\longrightarrow} \mathcal{N}(0,1),
$$
 as $N,T\rightarrow\infty$. Since,   by Lemma~\ref{conv-variance} and Proposition~\ref{Th2.2}, we have that
\begin{align*}
\lim_{N,T\to\infty}d_{TV}\left(R_{N,T}\widehat{F}_{N,T}, \mathcal{N}(0,1)\right)=0.
\end{align*}

Consequently, Theorem~\ref{Equiv} implies that
\begin{align}\label{weakconv}
    \w\lim_{N,T\to\infty}R_{N}\widehat{F}_{N}= \mathcal{N}\left(0,1\right).
\end{align}
Hence, by Slutsky's theorem, we deduce that $R_{N,T}(F_{N,T}-\widehat{F}_{N,T}) \overset{d}{\longrightarrow} 0$ ,  as $N,T\rightarrow\infty$, which consequently implies that
\begin{equation}\label{eq:M2}
R_{N,T}(F_{N,T}-\widehat{F}_{N,T}) \longrightarrow 0,
\end{equation}
 as $N,T\rightarrow\infty$, in probability.
And finally, \eqref{weakconv} combined with \eqref{eq:M1} and \eqref{eq:M2}, implies that
\begin{align*}
\w\lim_{N,T\to\infty}R_{N,T} F_{N,T}= \mathcal{N}(0,1).    
\end{align*}
\end{proof}

 And finally we can conclude with the asymptotic normality of the estimator.

\begin{corollary}
Under assumptions \eqref{fron} and \eqref{coef},  the next limit  holds
\begin{equation*}
\begin{split}
&\lim_{N,T\to \infty} T N^{3/2}(\widehat{\lambda}_{N,T}-\lambda)=
\mathcal{N}\left(0, 12\lambda\right),
\end{split}
\end{equation*}
in distribution.
\end{corollary}


\section{Asymptotic properties of the discretized MLE}
\label{sect-Discrete}

In this section, we investigate statistical properties of a discretized version of MLE. This properties are consistency and asymptotic normality in a weaker senses. Before the main results in this section, we prove several lemmas that we will use afterward.

We assume that the Fourier modes $u_k(t)$, $k\geq 1$, are observed at a uniform time grid
$$
0=t_0<t_1<\cdots < t_M =T,\quad \textrm{with}\ \Delta t:=t_i-t_{i-1}=\displaystyle\frac{T}{M},\
i=1,\ldots,M.
$$
We consider the discretized MLE $\widehat{\lambda}_{N,M}$ defined by
\begin{equation*}
\widehat{\lambda}_{N,M}:=-\frac{\sum_{k=1}^N k^2\sum_{i=1}^M
u_k(t_{i-1})\left[v_k(t_i)-v_k(t_{i-1})\right]}
{\sum_{k=1}^N k^4\sum_{i=1}^M u_k^2(t_{i-1})\Delta t}.
\end{equation*}
We are interested in studying the asymptotic properties of $\widehat{\lambda}_{N,M}$, as $N,M\to\infty$.

For simplicity of writing, we also introduce the following notations:
\begin{align*}
\xi_{N,M}&:=\sum_{k=1}^N k^2 \sum_{i=1}^M u_k(t_{i-1})\left(w_k(t_i)-w_k(t_{i-1})\right), &
\xi_{N}&=\sum_{k=1}^N k^2\int_0^T u_k(t)d w_k(t),\\
J_{N,M}&:=\sum_{k=1}^N k^4\sum_{i=1}^M u_k^2(t_{i-1})\Delta t,
& J_{N} & =\sum_{k=1}^N k^4\int_0^T u_k^2(t)d t,\\
V_{N,M}&:=\sum_{k=1}^N k^4 \sum_{i=1}^M u_k(t_{i-1})
\int_{t_{i-1}}^{t_i} \left(u_k(t)-u_k(t_{i-1})\right)dt, &  \Upsilon&:=  \left(\frac{T^2}{12\lambda}\right)^{1/2}.
\end{align*}

\subsection{Technical lemmas}
A key step in the proofs of the main results is to write $\widehat{\lambda}_{N,M}$ as
\begin{equation}\label{Dis-MLE-theta}
\widehat{\lambda}_{N,M}-\lambda = \frac{\lambda V_{N,M}}{J_{N,M}} -\frac{\sigma \xi_{N,M}}{J_{N,M}}.
\end{equation}

Now, we present some technical results whose proofs are presented in  \ref{Appendix:proofs}.

\begin{lemma}\label{lemma:Discr1}
For $0<t<s\leq T$ and $k, l\in \mathbb{N}$, we have that
\begin{align}
\mathbb{E}(u_k(t)u_k(s))& =\frac{\sigma^2}{2\ell_k^2} \left[t\cos (\ell_k (t-s))+\frac{1}{2}\left(\sin (\ell_k (t-s))-\sin (\ell_k (t+s)) \right)\right], \label{eq:productuk} \\
\mathbb{E}|u_k(t)-u_k(s)|^{2l}&\leq C(l)\left(\sigma^2\ell_k^{-1}\right)^l T^l |t-s|^l, \label{eq:continuityuk} \\
\mathbb{E}|u_k(t)+u_k(s)|^{2l}&\leq \bar C(l)\left(\sigma^2\ell_k^{-2}\right)^l T^l, \label{eq:bddnessuk}
\end{align}
for some $C(l), \bar C(l)>0$.
\end{lemma}

\begin{lemma}\label{lemma:Y-Difference}
For each $T>0,\ N,M\in \mathbb{N}$,
there exist constants $C>0$ independent of $M$ such that
\begin{align}
\mathbb{E}|\xi_{N,M}-\xi_{N}|^2 &\simeq C \frac{T^2 N^{4}}{M}, \label{eq:Y-NMT}\\
\mathbb{E}\left|J_{N,M}-J_{N}\right|^2 &\simeq  C \frac{T^5N^{6}}{M}, \label{eq:I-NMT}\\
\mathbb{E}|V_{N,M}|^2&\simeq C\frac{T^5N^6}{M}, \label{eq:V-NMT}
\end{align}
as $N\rightarrow \infty$.
\end{lemma}

We also present three well-known inequalities that we will use in the proof of the asymptotic properties of the discretized MLE.

\begin{proposition}
Let  $X,Y,Z$ be random variables, and assume that $Z>0$ a.s.. For any $\varepsilon>0$ and $\delta\in(0,\varepsilon/2)$, the following inequalities hold true.
\begin{align}
& \mathbb{P}(|Y/Z|>\varepsilon) \leq \mathbb{P}(|Y|>\delta) + \mathbb{P}(|Z-1|> (\varepsilon-\delta)/\varepsilon), \label{eq:simpleOne}\\
&  \sup_{x\in\mathbb{R}} \Big|  \mathbb{P}(X+Y\leq x) - \Phi(x)\Big|   \leq \sup_{x\in\mathbb{R}} \Big|  \mathbb{P}(X\leq x) - \Phi(x)\Big| + \mathbb{P}(|Y|>\varepsilon) + \varepsilon, \label{eq:simple2} \\
&  \sup_{x\in\mathbb{R}} \Big|  \mathbb{P}(Y/Z\leq x) - \Phi(x)\Big|   \leq \sup_{x\in\mathbb{R}} \Big|  \mathbb{P}(Y\leq x) - \Phi(x)\Big| + \mathbb{P}(|Z-1|>\varepsilon) + \varepsilon,  \label{eq:simple3}
\end{align}
where $\Phi$ denotes the distribution function of a standard Gaussian random variable.
\end{proposition}

\subsection{Asymptotic properties}

In this section we prove the asymptotic weak consistency  and an asymptotic normality result for the discretized estimator. 

\begin{remark}
On the proof of Theorem \ref{main-th1}, the authors on \cite{LiuL} show that
 \begin{equation}\label{eq:strongLLN-1}
\lim_{N \to\infty} \frac{\xi_{N}}{N^3\sigma^2\Upsilon^2} =0, \quad \mbox{and} \quad
\lim_{N \to\infty}\frac{J_{N}} {N^3\sigma^2\Upsilon^2}= 1,
\end{equation}
with probability one. Moreover, we have the following limit in distribution
\begin{align}
    \lim_{N\to \infty} \frac{\xi_{N,T}}{\sqrt{\bE J_{N,T}}}&=\mathcal{N}(0,1)\label{normal-1l2}
\end{align}
\end{remark}

With these at hand, we are ready to show that $\widehat{\lambda}_{N,M}$ is a weakly consistent estimator of $\lambda$.
\begin{theorem}\label{th:DiscrtConsistency}
Assume \eqref{coef} and \eqref{fron}. Then,
\begin{equation}\label{eq:thetaNTM-Consist}
\widehat{\lambda}_{N,M} \rightarrow \lambda, \quad \mbox{in probability},
\end{equation}
as  $N,M\to \infty$ while $T$ is fixed.
\end{theorem}

\begin{proof}
Let $\bar L := \mathbb{P}\left(\left|\widehat{\lambda}_{N,M}-\lambda\right|>\varepsilon\right)$.  In view of \eqref{Dis-MLE-theta}, we note that
\begin{align*}
\bar L
&\leq \mathbb{P}\left(\left|\frac{\sigma \xi_{N,M}}{J_{N,M}}\right|>\varepsilon/2\right)
+\mathbb{P}\left(\left|\frac{\lambda V_{N,M}}{J_{N,M}}\right|>\varepsilon/2\right).
\end{align*}
Consequently, for an arbitrary fixed $\delta\in(0,\varepsilon/2)$, we have, using \eqref{eq:simpleOne},
\begin{align*}
\bar L
&\leq \mathbb{P}\left(\frac{\lambda|V_{N,M}|}{N^3\sigma^2\Upsilon^2}>\delta\right)
+\mathbb{P}\left(\frac{|\xi_{N,M}|}{N^3\sigma\Upsilon^2}>\delta\right) \\
&\qquad+2\mathbb{P}\left(\left|\frac{J_{N,M}}{N^3\sigma^2\Upsilon^2}-1\right|>\frac{\varepsilon-2\delta}{\varepsilon}\right)\\
& =: L_1 +L_2 +2L_3.
\end{align*}
By Chebyshev's inequality, and \eqref{eq:V-NMT}, we have that, for some constant (that may depend on $\delta$) $C_1(\delta)>0$,
$$
L_1 \simeq C_1(\delta)\frac{T}{M},
$$
as $N\rightarrow \infty$. As far as $L_2$, we write
\begin{align*}
L_2 \leq \mathbb{P}\left(
\frac{|\xi_{N,M}-\xi_{N,T}|}{N^3\sigma\Upsilon^2}>\delta/2\right)+
\mathbb{P}\left(\frac{|\xi_{N,T}|}{N^3\sigma\Upsilon^2}>\delta/2\right) =: L_{21}+L_{22}.
\end{align*}
Again by Chebyshev inequality, and using \eqref{eq:Y-NMT}, we get $L_{21} \simeq C_2(\delta)/( T^2 N^2)$, as $N\rightarrow \infty$ for some $C_2(\delta)>0$. On the other hand, by \eqref{eq:strongLLN-1}, $L_{22}\to 0$, as $N\to \infty$. 
We treat $L_3$ similarly:
\begin{align*}
L_3\leq \mathbb{P}\left(\frac{|J_{N,M}-J_{N,T}|}{N^3\sigma^2\Upsilon^2}
>\frac{\varepsilon-2\delta}{2\varepsilon}\right)
+\mathbb{P}\left(\left|\frac{J_{N,T}}{N^3\sigma^2\Upsilon^2}-1\right|>\frac{\varepsilon-2\delta}{2\varepsilon}\right) =: L_{31}+L_{32}.
\end{align*}
In view of  \eqref{eq:I-NMT}, and Chebyshev inequality, we have the asymptotic behavior 
\begin{equation*}
L_{31}\simeq C_3(\varepsilon) \frac{T}{M},
\end{equation*}
as $N\rightarrow \infty$. By \eqref{eq:strongLLN-1}, we get that $L_{32}\to0$, as $N\to \infty$.
Hence, combining all the above bounds, we conclude that
\begin{equation*}
\bar L \simeq C(\varepsilon)\left(\frac{T}{ M} +
\frac{1}{T^2 N^2}\right),
\end{equation*}
as $N\rightarrow \infty$. Clearly, $\bar L\to 0$ for every $\varepsilon>0$,  as $N,M\to \infty$ when $T$ fixed.
This concludes the proof.
\end{proof}

We now focus on the proof of the asymptotic  normality result for discretized MLE $\widehat{\lambda}_{N,M}$.  It is not a surprise that the rate of convergence of $\widehat{\lambda}_{N,M}$ agrees with those from  continuous time setup, and thus asymptotically is optimal in the mean-square sense.  As usual, we denote by $\Phi$ the cumulative probability function of a standard Gaussian random variable.

\begin{theorem}\label{th:DiscrAsymNorm}
Assume \eqref{coef} and \eqref{fron}. Then,
\begin{equation}\label{eq:DiscrtAsymNorm1}
\sup_{x\in \mathbb{R}}\left|
\mathbb{P}\left(N^{3/2}\sigma\Upsilon\left(\lambda-\widehat{\lambda}_{N,M}\right)\leq x\right)-\Phi(x)
\right|\rightarrow 0,
\end{equation}
as $M\to \infty$ and $N$ sufficiently large while $T$ is fixed.
\end{theorem}

\begin{proof}
We denote the left hand side of \eqref{eq:DiscrtAsymNorm1} by $\bar K$, and using \eqref{Dis-MLE-theta}, we write it as
\begin{equation*} 
\bar K = \sup_{x\in \mathbb{R}}\left|
\mathbb{P}\left(N^{3/2}\sigma\Upsilon\frac{\sigma \xi_{N,M} - \lambda V_{N,M}}{J_{N,M}}
 \leq x\right)-\Phi(x)\right|.
\end{equation*}
Using \eqref{eq:simple3}, we continue
\begin{align*}
\bar K & \leq
\sup_{x\in \mathbb{R}}\left|
\mathbb{P}\left(\frac{\sigma \xi_{N,M} - \lambda V_{N,M}}{N^{3/2}\sigma\Upsilon}
 \leq x\right)-\Phi(x)\right| +
\mathbb{P}\left(\left|\frac{J_{N,M}}{N^{3}\sigma^2\Upsilon^2}-1\right|>\varepsilon \right)+\varepsilon  \\
& =: K_1 + K_2 +\varepsilon.
\end{align*}
Consequently, by applying \eqref{eq:simple2} , we obtain
\begin{align*}
  K_1 & \leq \sup_{x\in \mathbb{R}}\left| \mathbb{P}\left(\frac{\xi_{N,T}}{N^{3/2}\sigma\Upsilon}  \leq x\right)-\Phi(x) \right|
+\mathbb{P}\left(\frac{|\xi_{N,M}-\xi_{N,T}|}{N^{3/2}\sigma\Upsilon}
 >\varepsilon \right) \\
 & \qquad +  \mathbb{P}\left(\frac{\lambda|V_{N,M}|}{N^{3/2}\sigma\Upsilon}>\varepsilon \right) + 2\varepsilon
 =: K_{1,1} + K_{1,2} + K_{1,3} + 2\varepsilon.
\end{align*}
Be aware that  \eqref{normal-1l2} implies that
\begin{equation*}
\w\lim_{N\to \infty} \frac{\xi_{N,T}}
{N^{3/2}\sigma\Upsilon} =\mathcal{N}(0,1).
\end{equation*}
Thus, $K_{1,1}\to 0$, as $N\rightarrow \infty$.
By Chebyshev inequality and by \eqref{eq:Y-NMT} and  \eqref{eq:V-NMT}, we deduce
\begin{equation*}
K_{1,2}  \simeq C_1(\varepsilon) \frac{ N}{M}, \qquad K_{1,3} \simeq C_2(\varepsilon) \frac{T^{3} N^3}{M},
\end{equation*}
for some $C_1(\varepsilon), C_2(\varepsilon)>0$ and $N$ sufficiently large.

Similarly,
\begin{align*}
  K_2 & \leq \mathbb{P}\left(\left|\frac{J_{N,T}}{N^{3}\sigma^2 \Upsilon^2} - 1\right|>\frac{\varepsilon}{2} \right)
+\mathbb{P}\left(\frac{|J_{N,M}-J_{N,T}|}{N^{3}\sigma^2\Upsilon^2}
>\frac{\varepsilon}{2} \right)=: K_{2,1}+K_{2,2}.
\end{align*}
By \eqref{eq:strongLLN-1}, $K_{2,1}\to0$, as $N\to\infty$. On the other hand, by \eqref{eq:I-NMT}, $$
K_{2,1} \simeq C_3(\varepsilon) \frac{T^3 N^3}{M},
$$
for $N$ sufficiently large. Combining all the above, we conclude
$$
\bar K \leq C_4(\varepsilon)\left(\frac{N}{M} + \frac{T^{3} N^3}{M} \right),
$$
for $N$ sufficiently large. Since $\varepsilon>0$ was chosen arbitrarily, and since $C_5(\varepsilon)$ is independent of $M$,  we conclude that $\bar K\to 0$, as $M\to \infty$ for $N$ sufficiently large and $T$ fixed. The proof is complete.
\end{proof}


\section{Simulations}
\label{sect-sim}

\subsection{Simulation of the estimators}\label{simu}
This section is devoted to illustrate numerical compatutations of the solution of the equation and to show the properties of the discretized versions of the estimators proved in the second section. Furthermore, it serves as numerical verification of the results from the fourth section.

First, we introduce the discretized solution to the equation, then the discretization of both estimators and finally we simulate the discretized version of the speed estimator studied in the previous section.

First, we assume that the Fourier modes $\{u_k (t),v_k (t)\}_{k\geq 1}$, are observed at a uniform time grid
$$0=t_0<t_1<...<t_M =T,\quad \text{with} \Delta t:=t_i -t_{i-1 }=\frac{T}{M},i=1,...,M.$$
We consider the discretized MLE $\widehat{\lambda}_{N,M,T}$  defined in the last section.

Now, we will simulate the Fourier modes $\{u_k (t),v_k (t)\}$ for $1\leq k\leq N$. Remembering that $\{u_k (t),v_k (t)\}$ are solutions of the system \eqref{Osc1}, we can simulate that system by the Euler method  with the next expressions:
\begin{align*}
    u_k(t_{i+1})&=u_k(t_{i})+v_k(t_{i}) \Delta t,\\
    v_k(t_{i+1})&=v_k(t_{i})-\lambda k^2 u_k(t_{i})\Delta t+\sigma \big(w_k(t_{i+1})-w_k(t_i)\big),
\end{align*}
 for $i=1,..,M$ and $u_k (t_0)=0, v_k (t_0)=0$, where $\Delta t=t_{i+1}-t_i =\frac{T}{M}$ and $$w_k(t_{i+1})-w_k(t_i)\sim \mathcal{N}(0,\Delta t).$$ Note that the Milstein method gives us the same expressions.

Now, we can calculate a approximation of the solutions:
 \begin{align*}
u(t_i,x_j)=\sqrt{\frac{2}{\pi}}\sum_{k=1}^N u_k(t_i)\sin(kx_j),\\
v(t_i,x_j)= \sqrt{\frac{2}{\pi}}\sum_{k= 1}^N v_k(t_i)\sin(kx_j).
\end{align*}\\
Then, a pair of simulations of the solution could be seen in \ref{fig:sol}  where the parameters are $\lambda =10,1.5,0.5$, $\sigma =5$, $M=1000$, and $N=100$. The $x$-axis is the space coordinate,the $y$-axis is the time coordinate, and the he $x$-axis is the solution. Notice  the larger the parameter the faster the solution increases. 

\begin{figure}
\subfloat[$\lambda=10$]{\includegraphics[width = 7cm]{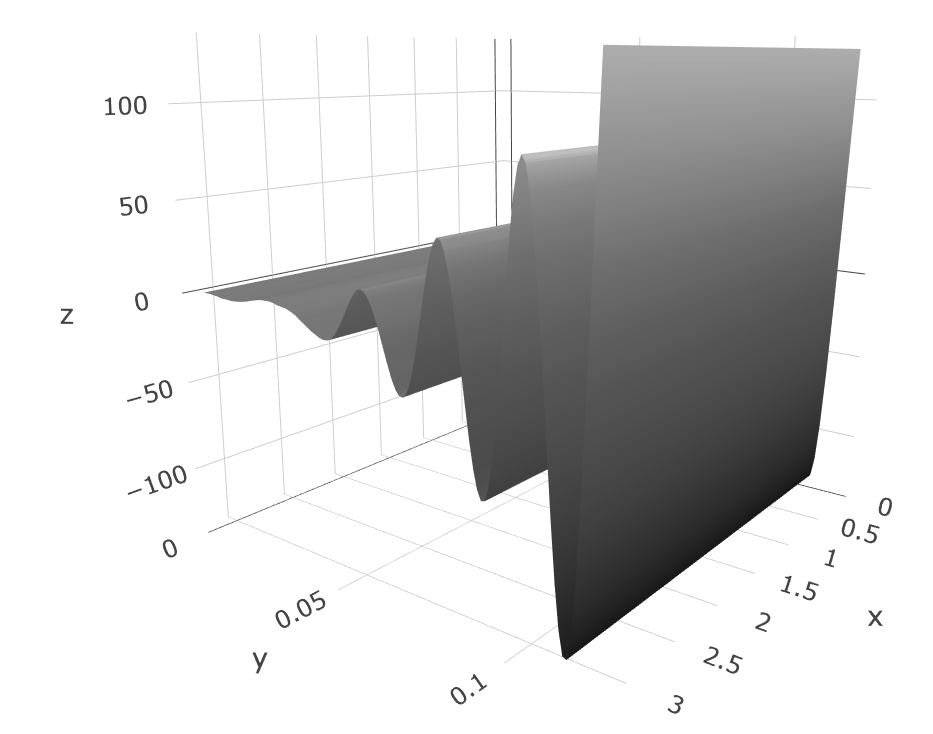}} 
\subfloat[$\lambda=0.5$]{\includegraphics[width = 7cm]{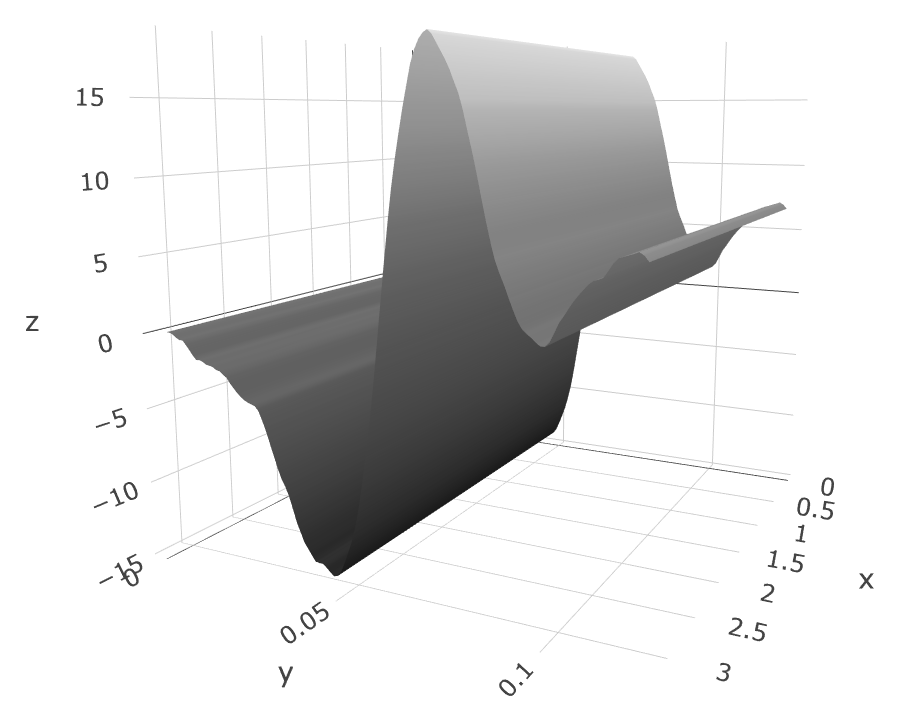}}\\
\subfloat[$\lambda=0.5$]{\includegraphics[width = 7cm]{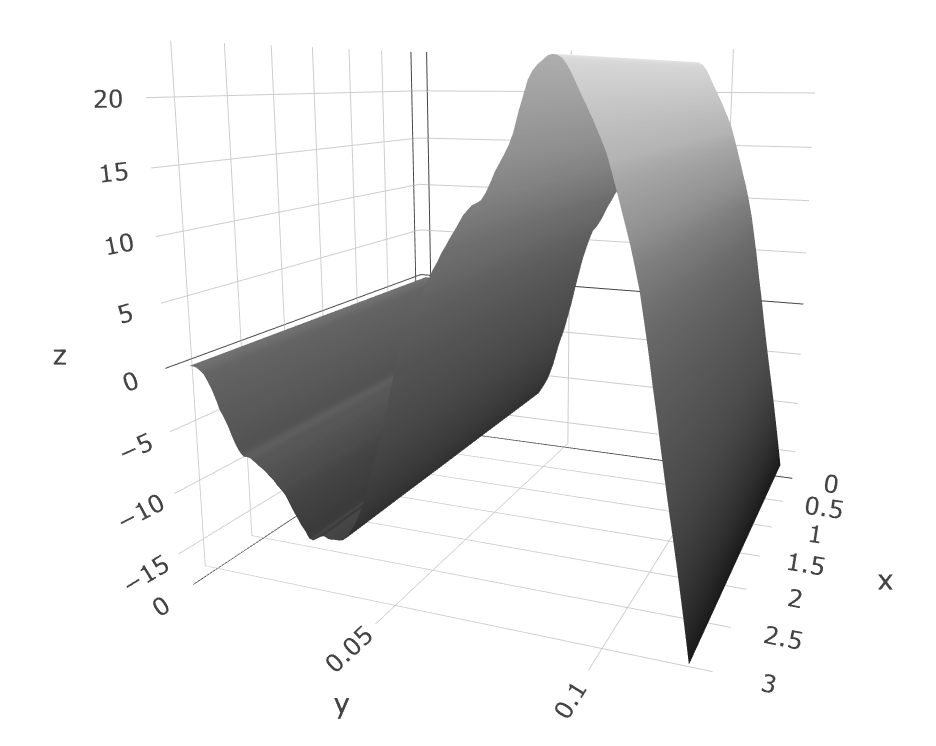}}
\caption{Solutions with distinct speed parameter}
\label{fig:sol}
\end{figure}

 At this point, we can calculate all the terms in \eqref{Dis-MLE-theta} and obtain the discretized MLE, for $N$, $M$ and $\sigma$ fixed. Thus, we present two examples of the simulation of the solutions of the equation and the discretized MLEs. First we fixed $T=1$, $M=1000$ and $N=100$.  

In \ref{fig:con},  the parameters are $N=100$, $M=10000$ with  $\lambda =5,20$  and $\sigma =0.8,3$ respectively. Note in both parameters we can visualize the consistency of both cases.
\begin{figure}
\subfloat[$\lambda=5,\sigma=0.8$]{\includegraphics[width = 6cm]{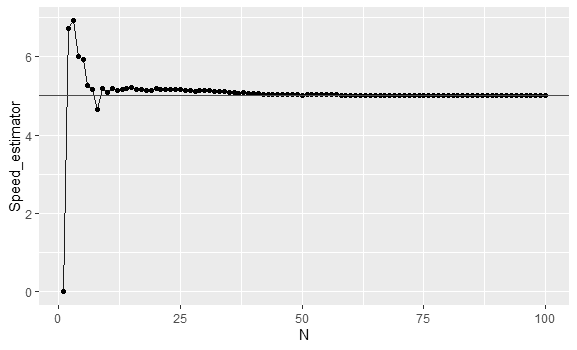}} 
\subfloat[$\lambda=20,\sigma=3$]{\includegraphics[width = 6cm]{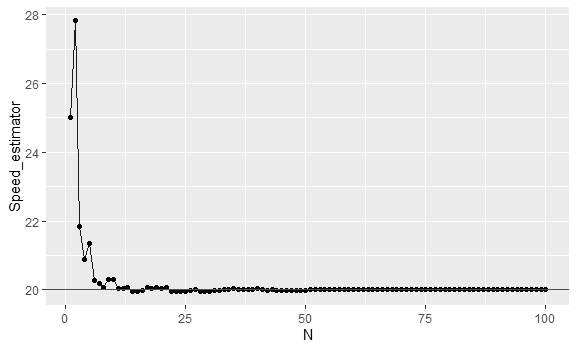}}
\caption{Asymptotic consistency as $N\rightarrow \infty$}
\label{fig:con}
\end{figure}\\
 And then we calculate $100$ estimations in \ref{fig:hist} with the same parameters respectively  and make two histograms of $N^{\frac{3}{2}}\sigma \Upsilon(\widehat{\lambda}_{1,N,M}-\lambda)$. Note the quasi-normality of the 100 estimations.
\begin{figure}
\subfloat[$\lambda=5,\sigma=0.8$]{\includegraphics[width = 6cm]{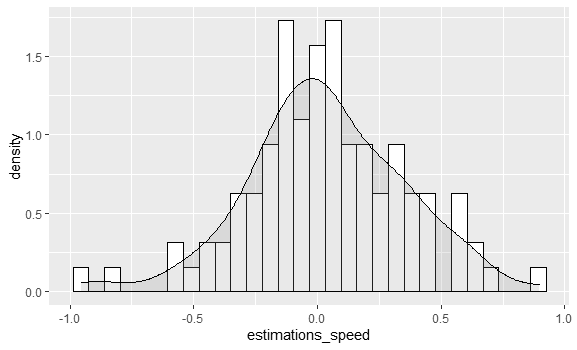}} 
\subfloat[$\lambda=20,\sigma=3$]{\includegraphics[width = 6cm]{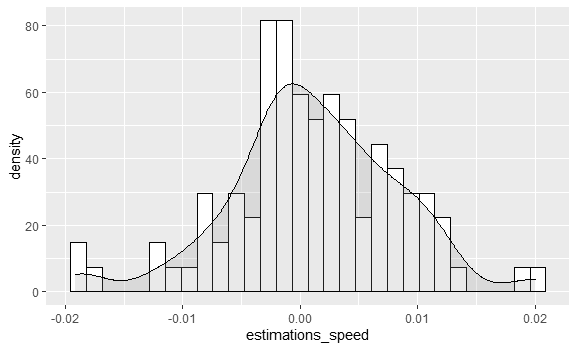}}
\caption{Asymptotic normality for $N$ sufficiently large}
\label{fig:hist}
\end{figure}
And finally we present three histograms in figure \ref{fig:weakasym} for different number of $M$ (partition fineness). In this case, the next parameters are fixed, $T=1$, $N=100$, $\lambda=1$ and $\sigma=0.5$ and it was  used a dyadic partition, i.e., $M=2^{8}, 2^{11}, 2^{15}$.  Note that as theorem \ref{th:DiscrAsymNorm}  states, the finer the partition, the better the empirical distribution resembles to a normal distribution.
\begin{figure}[ht!]
\subfloat[Histogram for $M=2^{8}=256$]{\includegraphics[width=8cm]{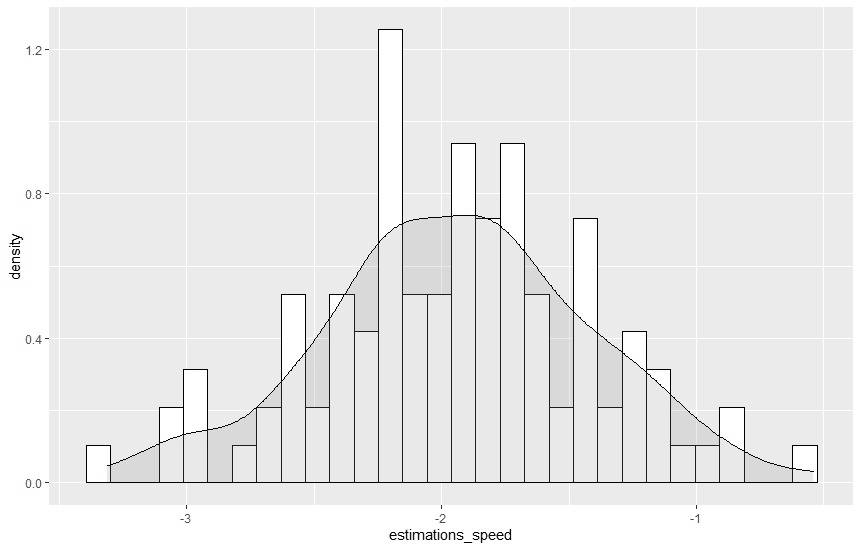}}
\subfloat[Histogram for $M=2^{11}=2048$]{\includegraphics[width=8cm]{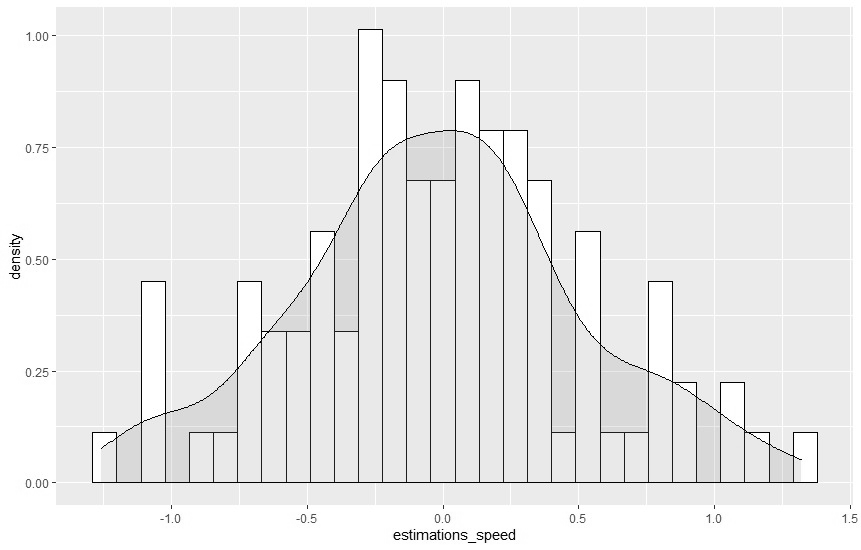}}\\
\subfloat[Histogram for $M=2^{11}=32768$]{\includegraphics[width=8cm]{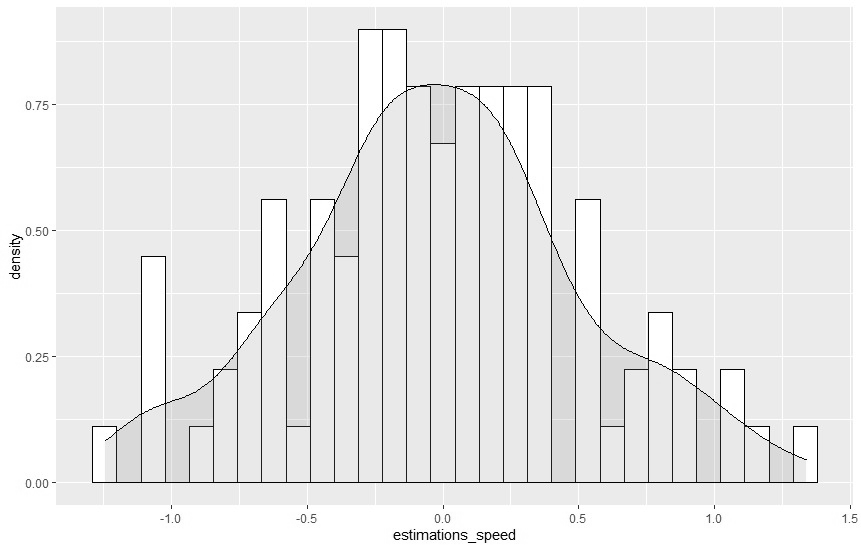}}
\caption{Histogram for estimations as $M\rightarrow\infty$. }
\label{fig:weakasym}
\end{figure}
\FloatBarrier


\appendix

\section{Proofs of technical lemmas}
\label{Appendix:proofs}
We begin with a lemma which is not difficult to prove.
\begin{lemma}\label{lemma-sqr}
For any process $\Phi=\{\Phi_s\}_{s\in[0,t]}$ such that $\sqrt{\Var (\Phi_s)}$ is integrable on $[0,t]$, it holds that
\begin{align*}
    \sqrt{\Var \left( \int_{0}^t \Phi_s ds \right)} \leq \int_{0}^t \sqrt{\Var (\Phi_s)} ds
\end{align*}
\end{lemma}

\begin{proof}[Proof of Lemma~\ref{conv-variance}]
We start by computing the Malliavin derivative of  $\widehat{F}_{N}$. If $r\leq t$ and for $1\leq k\leq N$, then
\begin{align*}
D_{r,k}u_k(t) &= \frac{\sigma }{\ell_k} D_{r,k}  \int_0^t\sin\big(\ell_k(t-s)\big)dw_k(s)\\
&= \frac{\sigma }{\ell_k}   \sin\big(\ell_k(t-r)\big).
\end{align*}
Moreover, one has that $D_{r,j}u_k(t)=0$ if $j\neq k$ or $r>t$.
Therefore, for $r\leq T$ and $1\leq j\leq N$, we have by \eqref{chain-rule-formula},
\begin{align}
D_{r,j}\widehat{F}_{N,T}&= \frac{1}{R_{N,T}^2}  j^2 u_j(r)+ \frac{1}{R_{N,T}^2}  \sum_{k=1}^N k^2\int_r^{T} D_{r,j}u_k(t)d w_k(t) \nonumber\\
&= \frac{1}{R_{N,T}^2}  j^2 u_j(r)+\frac{1}{R_{N,T}^2}   j^2 \int_r^{T} D_{r,j}u_j(t)d w_j(t) \nonumber\\
&= \frac{1}{R_{N,T}^2}  j^2 u_j(r)+ \frac{\sigma}{\ell_j R_{N,T}^2}   j^2 \int_r^{T} \sin\big(\ell_j(t-r)\big) d w_j(t). \label{est12}
\end{align}
We continue by setting
$$
A:=\left\| R_{N,T} D\widehat{F}_{N,T}\right\|_{\mathcal{H}}^2 =  R_{N,T}^2\| D\widehat{F}_{N,T}\|_{\mathcal{H}}^2,
$$
and in view of \eqref{est12}, we obtain
\begin{align*}
A&=R_{N,T}^2\int_0^T \sum_{k=1}^N \left[\frac{1}{R_{N,T}^2}  k^2 u_k(r)+ \frac{\sigma}{\ell_k R_{N,T}^2}   k^2 \int_r^{T}  \sin\big(\ell_k(t-r)\big) d w_k(t)\right]^2 d r \nonumber\\
&=  \int_0^{T}\sum_{k=1}^N \Bigg[\frac{1}{R_{N,T}^2}
k^4 u_k^2(r)
+  2\frac{\sigma}{\ell_k R_{N,T}^2} k^4  u_k(r)  \int_r^{T}   \sin\big(\ell_k(t-r)\big) d w_k(t)\\ \nonumber
&\qquad \qquad \qquad \qquad + \frac{\sigma^2}{\ell_k^2 R_{N,T}^2} k^4
\left(\int_r^{T}   \sin\big(\ell_k(t-r)\big) d w_k(t)\right)^2 \Bigg]d r \nonumber\\
&=:A_1+ A_2+A_3. 
\end{align*}
By Lemma \ref{lemma-sqr}, we have
\begin{align*}
\sqrt{ \Var \left(\frac{1}{2} \| R_{N,T} D\widehat{F}_{N,T}\|_{\mathcal{H}}^2 \right) } &\leq \frac{\sqrt{3}}{2}\left(  \sqrt{ \Var(A_1) } +  \sqrt{ \Var(A_2)  } +  \sqrt{\Var(A_3)  } \right) \\
&\leq  \frac{\sqrt{3}}{2}\left(B_1+B_2+B_3\right),
\end{align*}
where
\begin{align*}
B_1&:=\frac{1}{R_{N,T}^2}\int_0^{T} \left[\Var\left( \sum_{k=1}^N k^4 u_k^2(r)\right) \right]^{1/2} d r\\
B_2&:=2\frac{\sigma}{ R_{N,T}^2}\int_0^{T} \left[\Var \left( \sum_{k=1}^N k^4 \ell_k^{-1}  u_k(r)  \int_r^{T}   \sin\big(\ell_k(t-r)\big) d w_k(t)\right) \right]^{1/2} d r\\
B_3&:= \frac{\sigma^2}{R_{N,T}^2} \int_0^{T} \left[\Var \left( \sum_{k=1}^N k^4 \ell_k^{-2}
\left(\int_r^{T}   \sin\big(\ell_k(t-r)\big) d w_k(t)\right)^2 \right)  \right]^{1/2} d r.
\end{align*}
For $B_1$, by the independence of $\{u_k\}_{k\ge 1}$, 
 \begin{align*}
B_1&=\frac{1}{R_{N,T}^2}\int_0^{T} \left(\sum_{k=1}^N k^8 \Var\left(u_k^2(r)\right) \right)^{1/2}d r,
\end{align*}
and remember that $\sum_{k=1}^N k^8  \Var (u_k^2 (t))\sim N^5 \frac{\sigma^4 t^2}{10\lambda^2}$, then we have
\begin{align*}
    \int_0^{T} \left[  \sum_{k=1}^N k^8 \Var\left(u_k^2(r)\right) \right]^{1/2} d r\simeq N^{\frac{5}{2}}\frac{\sigma^2}{\sqrt{10}\lambda}\int_0^T t dr=N^{\frac{5}{2}}\frac{\sigma^2 T^2}{2\sqrt{10}\lambda},\quad \text{as } N\rightarrow\infty
\end{align*}
Thus, by  \ref{lem-1},
\begin{align}
    B_1&=\frac{1}{R_{N,T}^2}\int_0^{T} \left(\sum_{k=1}^N k^8 \Var\left(u_k^2(r)\right) \right)^{1/2}d r\nonumber\\
    &\simeq \frac{N^{\frac{5}{2}}\frac{\sigma^2 T^2}{2\sqrt{10}\lambda}}{N^3 \frac{\sigma^3 
    T^2}{12\lambda}}=C\frac{1}{N^{\frac{1}{2}}}\rightarrow 0 \quad \mbox{as}\ N,T\to \infty. \label{A1_conv2}
\end{align}
For $B_2$, we note that $u_k$ and $w_l$ are independent if $k\neq l$.
Therefore, we rewrite $B_2$ as
\begin{align*}
B_2&=\frac{2\sigma}{R_{N,T}^2}\int_0^{T} \left[ \sum_{k=1}^N \frac{k^8}{\ell_k^2}\Var \left(u_k(r) \int_r^{T}  \sin\big(\ell_k(t-r)\big)d w_k(t) \right)\right]^{1/2} d r.
\end{align*}
By straightforward calculations, we have that
\begin{align*}
\Var \left(u_k(r) \int_r^{T} \sin\big(\ell_k(t-r)\big) d w_k(t) \right)
&\leq \bE\left[ u_k^2(r)\left(   \int_r^{T}  \sin\big(\ell_k(t-r)\big) d w_k(t)\right)^2 \right]\\
& = \bE\Bigg[ u^2_k(r)\\
&\hspace{1cm}\bE\Big[ \left(  \int_r^{T}  \sin\big(\ell_k(t-r)\big)  d w_k(t)\right)^2\Big| \mathscr{F}_r \Big] \Bigg]\\
&= \bE[ u^2_k(r)]  \int_0^{T-r}  \sin^2\big(\ell_k t\big) d t.
\end{align*}
Note that 
\begin{align*}
    \lim_{k\rightarrow\infty} \frac{k^4}{\ell_k^2} \bE[ u^2_k(r)]  \int_0^{T-r}   \sin^2\big(\ell_k t\big) d t&=
    \frac{\sigma^2}{4\lambda^4}r(T-r) .
\end{align*}
Thus,
\begin{align*}
    \int_0^{T} \left[ \sum_{k=1}^N k^4 \frac{k^4}{\ell_k^2} \bE[ u^2_k(r)]  \int_0^{T-r}   \sin^2\big(\ell_k t\big) d t \right]^{1/2} d r&\simeq N^{\frac{5}{2}}\frac{\sigma}{5\lambda^2} \int_0^T \Big( r(T-r)\Big)^\frac{1}{2}dr\\
    &=N^{\frac{5}{2}}\frac{\sigma T^2}{40\lambda^2},\quad \mbox{as}\ N\to \infty.
\end{align*}
 Thus, by  \ref{lem-1},
\begin{align}
    B_2&\leq \frac{2\sigma}{R_{N,T}^2}\int_0^{T} \left[ \sum_{k=1}^N \frac{k^8}{\ell_k^2}\bE[ u^2_k(r)]  \int_0^{T-r}   e^{\lambda_2 t}\sin^2\big(\ell_k t\big) d t\right]^{1/2} d r\nonumber\\
    &\simeq \frac{\sigma N^{\frac{5}{2}}}{N^3 \frac{\sigma^2 T^2}{12\lambda}}=C\frac{1}{N^{\frac{1}{2}}}\rightarrow 0 \quad \mbox{as}\ N,T\to \infty. \label{A2_conv2}
\end{align}
Let us now consider $B_3$. Since $w_k$ and $w_j$ are independents for $k\ne j$, we have
 \begin{align*}
B_3&= \frac{\sigma^2}{R_{N,T}^2} \int_0^{T} \left[\Var \left( \sum_{k=1}^N k^4\ell_k^{-2}   \left(\int_r^{T}    \sin\big(\ell_k(t-r)\big) d w_k(t)\right)^2\right) \right]^{1/2} d r\\
&=   \frac{\sigma^2}{R_{N}^2} \int_0^{T} \left[ \sum_{k=1}^N k^8\ell_k^{-4} \Var \left(\int_r^{T}    \sin\big(\ell_k(t-r)\big)
d w_k(t)\right)^2 \right]^{1/2} d r\\
&\leq  \frac{\sigma^2}{R_{N}^2} \int_0^{T} \left[ \sum_{k=1}^N k^8\ell_k^{-4} \bE \left(\int_r^{T}    \sin\big(\ell_k(t-r)\big)
d w_k(t)\right)^4 \right]^{1/2} d r\\
&=\frac{\sigma^2}{R_{N}^2} \int_0^{T} \left[3 \sum_{k=1}^N k^8\ell_k^{-4}  \left(\int_0^{T-r}    \sin^2\big(\ell_k t\big)
d t\right)^2 \right]^{1/2} d r.
\end{align*}
Notice that 
\begin{align*}
    \lim_{k\rightarrow\infty} \frac{k^4}{\ell_k^4} \left( \int_0^{T-r}  \sin^2\big(\ell_k t\big) d t\right)^2&=
    \frac{\sigma^2}{4\lambda^2}(T-r)^2.
\end{align*}
Thus,
\begin{align*}
     \int_0^{T} \left[ 3\sum_{k=1}^N k^4 \frac{k^4}{\ell_k^4}   \left( \int_0^{T-r}   \sin^2\big(\ell_k t\big) d t\right)^2 \right]^{1/2} d r&\simeq N^{\frac{5}{2}}\frac{\sigma^2}{4\sqrt{\frac{3}{5}}\lambda^2}\int_0^T (T-r) dr,\\
     &=N^{\frac{5}{2}}\frac{\sigma T^2}{2\sqrt{\frac{3}{5}}\lambda^2}, \quad \mbox{as}\ N\to \infty.
\end{align*}
 Thus, by  \ref{lem-1},
\begin{align}
    B_3&\leq \frac{\sigma^2}{R_{N,T}^2} \int_0^{T} \left[3 \sum_{k=1}^N k^8\ell_k^{-4}  \left(\int_0^{T-r}    \sin^2\big(\ell_k t\big)
d t\right)^2 \right]^{1/2} d r.\nonumber\\
    &\simeq \frac{N^{\frac{5}{2}}\frac{\sigma T^2}{2\sqrt{\frac{3}{5}}\lambda^2}}{N^3 \frac{\sigma^3 T^2}{12\lambda}}=C\frac{1}{N^{\frac{1}{2}}}\rightarrow 0 \quad \mbox{as}\ N,T\to \infty. \label{A3_conv2}
\end{align}
Finally, combining \eqref{A1_conv2}, \eqref{A2_conv2} and \eqref{A3_conv2}, we have that for every $\varepsilon>0$,
there exist two independent constants $N_0,T_0>0$ such that for all $N\geq N_0$ and $T\geq T_0$,
$$
B_1+B_2+B_3  < \varepsilon.
$$
This completes the proof.
\end{proof}

\begin{proof}[Proof of Lemma~\ref{lemma:Discr1}]
By the It\^{o}'s isometry, we have
\begin{align*}
    \bE[u_k (t)u_k (s)]&=\frac{\sigma^2}{\ell_k^2}\int_0^t \sin (\ell_k (t-r))\sin (\ell_k (s-r))dr\\
    &=\frac{\sigma^2}{2\ell_k^2}\int_0^t \left[\cos (\ell_k (t-s))-\cos (\ell_k (t+s-2r)) \right]dr\\
    &=\frac{\sigma^2}{2\ell_k^2} \left[t\cos (\ell_k (t-s))-\int_0^t\cos (\ell_k (t+s-2r)) dt\right]\\
    &=\frac{\sigma^2}{2\ell_k^2} \left[t\cos (\ell_k (t-s))+\frac{1}{2\ell_k}\left(\sin (\ell_k (s-t))-\sin (\ell_k (s+t)) \right)\right],\\
\end{align*}
for $s<t$. As far as \eqref{eq:continuityuk} and \eqref{eq:bddnessuk}, since $u_k(t)-u_k(s)$ and $u_k(t)+u_k(s)$ is a Gaussian random variable, it is enough to prove \eqref{eq:continuityuk} and \eqref{eq:bddnessuk} for $l=1$. We note that
for $t<s$,
\begin{align*}
\mathbb{E}|u_k(t)-u_k(s)|^2&=\left|\bE u_k^2 (t)+\bE u_k^2 (s)-2\bE[u_k (t)u_k(s)] \right|\\
&=\frac{\sigma^2}{\ell_k^2}\Bigg|\frac{t}{2}-\frac{\sin(2\ell_k t)}{4\ell_k}+\frac{s}{2}-\frac{\sin(2\ell_k s)}{4\ell_k}-t\cos (\ell_k (t-s))\\
&\hspace{3cm}-\frac{1}{2\ell_k}\left(\sin (\ell_k (s-t))-\sin (\ell_k (s+t)) \right)\Bigg|\\
&=\frac{\sigma^2}{\ell_k^2}\Bigg(\Big|t(1-\cos (\ell_k (t-s)))\Big|+\frac{1}{2}|t-s|+\\
&\hspace{1cm}\Bigg|\sin(\ell_k (t+s))\Big(\frac{1}{2\ell_k}-\frac{\cos(\ell_k(t-s))}{2\ell_k}\Big)\Bigg|+\frac{1}{2\ell_k}\left|\sin (\ell_k (t-s)) \right|\Bigg)\\
&\leq \frac{\sigma^2}{\ell_k^2}\Bigg(t\Big|(\cos(0)-\cos (\ell_k (t-s)))\Big|+\frac{1}{2}|t-s|\\
&\hspace{2cm}+\frac{1}{2}|\cos(0)-\cos(\ell_k(t-s))|+\frac{1}{2} |t-s|\Bigg)\\
&\leq \frac{\sigma^2}{\ell_k^2}\Bigg(t\ell_k |t-s|+\frac{1}{2}(1+\ell_k)|t-s|+\frac{1}{2}|t-s|\Bigg)\leq C\sigma^2\ell_k^{-1}T|t-s|,
\end{align*}
for some $C>0$, and where in the last line we used the fact that $\sin(x)$ and $\cos(x)$ are Lipschitz. Thus part \eqref{eq:continuityuk} is proved.
Similarly for $\mathbb{E}|u_k(t)+u_k(s)|^2$ and the proof is complete.
\end{proof}

\begin{proof}[Proof of Lemma~\ref{lemma:Y-Difference}]
Since $u_k, \ k \geq 1$, are independent, and taking into account that
$$
\int_0^T u_k(t)d w_k(t)=\sum_{i=1}^M \int_{t_{i-1}}^{t_i}u_k(t)d w_k(t),
$$
we have that
\begin{align*}
\mathbb{E}|\xi_{N,M}-\xi_{N}|^2 &
=\mathbb{E}\left|\sum_{k=1}^N k^2 \left[\sum_{i=1}^M u_k(t_{i-1})\left(
w_k(t_i)-w_k(t_{i-1})\right)-\int_0^T u_k(t)d w_k(t)\right] \right|^2\\
&=\sum_{k=1}^N k^4 \mathbb{E}
\left|\sum_{i=1}^M u_k(t_{i-1})\left( w_k(t_i)-w_k(t_{i-1})\right)-\int_0^T u_k(t)d w_k(t)\right|^2 \\
&=\sum_{k=1}^N k^4 \sum_{i=1}^M \int_{t_{i-1}}^{t_i} \mathbb{E}\left(u_k(t_{i-1})-u_k(t)\right)^2d t
\end{align*}
and hence, by \eqref{eq:continuityuk}, there exist constants $C_1,C_2>0$, such that
\begin{align*}
\mathbb{E}|\xi_{N,M}-\xi_{N}|^2
&\leq C_1\sum_{k=1}^N \frac{k^4}{\ell_k} \sum_{i=1}^M \int_{t_{i-1}}^{t_i} T|t_{i-1}-t| d t\\
&= C_1\left(\sum_{k=1}^N \frac{k^4}{\ell_k} \right)\frac{T^3}{M}
\simeq C_2T^2 \frac{ N^{4}}{M}.
\end{align*}
as $N\rightarrow \infty$. Hence, \eqref{eq:Y-NMT} follows at once.

Next we will prove \eqref{eq:I-NMT}. We note that
\begin{align*}
\mathbb{E}|J_{N,M}-J_{N}|^2&=
\mathbb{E}\left|\sum_{k=1}^N k^4\left(\sum_{i=1}^M u_k^2(t_{i-1})(t_i-t_{i-1})
-\int_0^T u_k^2(t)d t\right)\right|^2\\
&=\sum_{k=1}^N k^8 \mathbb{E}\left|
\sum_{i=1}^M \int_{t_{i-1}}^{t_i} \left(u_k^2(t_{i-1})-u_k^2(t)\right)d t \right|^2.
\end{align*}
Consequently, letting $U_i(t):=u_k^2(t_{i-1})-u_k^2(t), \ k\geq 1$, we continue
\begin{align*}
\mathbb{E}|J_{N,M}-J_{N}|^2
&=\sum_{k=1}^N k^8 \sum_{i=1}^M \mathbb{E}\left|
 \int_{t_{i-1}}^{t_i} U_i(t)d t
\right|2+2\sum_{k=1}^N k^8 \sum_{i<j}\mathbb{E}
 \int_{t_{i-1}}^{t_i}\int_{t_{j-1}}^{t_j} U_i(t)U_j(s) d sd t\\
&=: I_1+I_2.
\end{align*}
Notice that by Cauchy-Schwartz inequality,
\begin{align*}
\mathbb{E}|U_i^2(t)|&=\mathbb{E}|u_k^2(t_{i-1})-u_k^2(t)|^2
=\mathbb{E}|u_k(t_{i-1})-u_k(t)|^2|u_k(t_{i-1})+u_k(t)|^2\\
&\leq \left(\mathbb{E}|u_k(t_{i-1})-u_k(t)|^4\right)^{1/2}
\left(\mathbb{E}|u_k(t_{i-1})+u_k(t)|^4\right)^{1/2}.
\end{align*}
Moreover, by \eqref{eq:continuityuk} and \eqref{eq:bddnessuk},
\begin{equation}
\label{Usquare}
\mathbb{E}|U_i^2(t)|\leq c_1\ell_k^{-3}T^2|t-s|,\quad \mbox{for some}\ c_1>0.
\end{equation}
Again by Cauchy-Schwartz inequality and \eqref{Usquare}, we have that
\begin{align*}
I_1&=\sum_{k=1}^N k^8 \sum_{i=1}^M \mathbb{E}\left|  \int_{t_{i-1}}^{t_i} U_i(t)d t \right|^2
\leq \sum_{k=1}^N k^8 \sum_{i=1}^M  (t_i-t_{i-1})
\int_{t_{i-1}}^{t_i} \mathbb{E}|U_i^2(t)|d t\\
&\leq c_1T^2\sum_{k=1}^N k^8\lambda_k^{-3} \sum_{i=1}^M  (t_i-t_{i-1})^3
=c_1\sum_{k=1}^N k^8\ell_k^{-3} \frac{T^5}{M^2}.
\end{align*}
Turning to $I_2$, we first notice that
\begin{align*}
\mathbb{E}|U_i(t)U_j(s)|&=\mathbb{E}\Big[\left(u_k^2(t_{i-1})-u_k^2(t) \right)
\left(u_k^2(t_{j-1})-u_k^2(s) \right)\Big]\\
&=\mathbb{E}\Big[\left(u_k(t_{i-1})-u_k(t)\right)
\left(u_k(t_{i-1})+u_k(t)\right)\\
&\hspace{1cm}\left(u_k(t_{j-1})-u_k(s)\right)\left(u_k(t_{j-1})+u_k(s)\right)\Big]\\
&=\mathbb{E}\Big[\left(u_k(t_{i-1})-u_k(t)\right)\left(u_k(t_{j-1})-u_k(s)\right)u_k(t_{i-1})u_k(t_{j-1})\Big]\\
&+\mathbb{E}\Big[\left(u_k(t_{i-1})-u_k(t)\right)\left(u_k(t_{j-1})-u_k(s)\right)(u_k(t_{i-1})u_k(s)\Big]\\
&+\mathbb{E}\Big[\left(u_k(t_{i-1})-u_k(t)\right)\left(u_k(t_{j-1})-u_k(s)\right)u_k(t)u_k(t_{j-1})\Big]\\
&+\mathbb{E}\Big[\left(u_k(t_{i-1})-u_k(t)\right)\left(u_k(t_{j-1})-u_k(s)\right)u_k(t)u_k(s)\Big].
\end{align*}
By the Wick's Lemma\cite[Theorem 3.1]{Bis}, we continue
\begin{align*}
\mathbb{E}|U_i(t)U_j(s)|
&=\mathbb{E}\left[\left(u_k(t_{i-1})-u_k(t)\right)
\left(u_k(t_{i-1})+u_k(t)\right)\right]\\
&\hspace{2cm}\mathbb{E}\left[\left(u_k(t_{j-1})-u_k(s)\right)\left(u_k(t_{j-1})+u_k(s)\right)\right]\\
&+\mathbb{E}\left[\left(u_k(t_{i-1})-u_k(t)\right)\left(u_k(t_{j-1})-u_k(s)\right)
\right]\\
&\hspace{2cm}\mathbb{E}\left[\left(u_k(t_{i-1})+u_k(t)\right)\left(u_k(t_{j-1})+u_k(s)\right)\right]\\
&+\mathbb{E}\left[\left(u_k(t_{i-1})-u_k(t)\right)\left(u_k(t_{j-1})+u_k(s)\right)
\right]\\
&\hspace{2cm}\mathbb{E}\left[\left(u_k(t_{i-1})+u_k(t)\right)\left(u_k(t_{j-1})-u_k(s)\right)\right]\\
&=:J_1+J_2+J_3.
\end{align*}
For $J_2$, we have
\begin{align*}
\mathbb{E}\left(u_k(t_{i-1})-u_k(t)\right)\left(u_k(t_{j-1})-u_k(s)\right)&=
\mathbb{E}\left(u_k(t_{i-1})u_k(t_{j-1})\right)
-\mathbb{E}\left(u_k(t_{i-1})u_k(s)\right)\\
&-\mathbb{E}\left(u_k(t)u_k(t_{j-1})\right)
+\mathbb{E}\left(u_k(t)u_k(s)\right).
\end{align*}
By \eqref{eq:productuk}, for $i<j$ and $t<s$($t_{i-1}\leq t\leq t_{j-1}\leq s$),
\begin{align*}
&\mathbb{E}\left(u_k(t_{i-1})-u_k(t)\right)\left(u_k(t_{j-1})-u_k(s)\right)=
\frac{\sigma^2}{2\ell_k^2}
\Bigg[t_{i-1}\cos (\ell_k (t_{i-1}-t_{j-1}))\\
&\ +\frac{1}{2\ell_k}\big(\sin (\ell_k (t_{i-1}-t_{j-1}))\nonumber -\sin (\ell_k (t_{i-1}+t_{j-1})) \Big) -t_{i-1}\cos (\ell_k (t_{i-1}-s))\nonumber\\
&\ +\frac{1}{2\ell_k}\left(\sin (\ell_k (t_{i-1}-s))-\sin (\ell_k (t_{i-1}+s)) \right)-t\cos (\ell_k (t-t_{j-1}))\nonumber\\
&\ +\frac{1}{2\ell_k}\left(\sin (\ell_k (t-t_{j-1}))-\sin (\ell_k (t+t_{j-1})) \right) +t\cos (\ell_k (t-s))+\frac{1}{2\ell_k}\big(\sin (\ell_k (t-s))\nonumber\\
&\ -\sin (\ell_k (t+s)) \Big)\Bigg] \nonumber\\
&\leq c_2 \ell_k^{-2}T,
\end{align*}
for some $c_2>0$. By similar arguments, we also obtain
$$
\mathbb{E}\left(u_k(t_{i-1})+u_k(t)\right)\left(u_k(t_{j-1})+u_k(s)\right)\leq
c_3\ell_k^{-2}T,
$$
for some $c_3>0$. Thus,
$$
J_2 \leq c_4\ell_k^{-4}T^2,
$$
for some $c_4>0$. By analogy, one can treat $J_1$ and $J_3$, and derive the following upper bounds:
$$
J_1\leq c_5\ell_k^{-4}T^2, \qquad
J_3 \leq c_6\ell_k^{-4}T^2,
$$
for some $c_5,c_6>0$. Finally, combining the above, we have
\begin{align*}
I_2&\leq c_7 T^2\sum_{k=1}^N k^8 \ell_k^{-4} \sum_{i<j} \int_{t_{j-1}}^{t_j}\int_{t_{i-1}}^{t_i}d t d s \\
&\leq c_8\sum_{k=1}^N k^8 \ell_k^{-4} \frac{T^4}{M},
\quad \mbox{for some}\ c_7,c_8>0.
\end{align*}
Thus, using the estimates for $I_1,I_2$, we conclude that
$$
I_1+I_2 \leq  c_9\sum_{k=1}^N k^8 \Big(\ell_k^{-3}\frac{T^5}{M^2}+\ell_k^{-4} \frac{T^4}{M} \Big)\simeq C_2 \frac{N^6 T^5 }{M},
$$
as $N\rightarrow \infty$ and hence \eqref{eq:I-NMT} is proved.
The estimate \eqref{eq:V-NMT} is proved by similar arguments.
\end{proof}

\section{Elements of Malliavin Calculus}\label{apendix:MalCal}
For the sake of completeness, in this section we recall some facts from Malliavin calculus associated with a Gaussian process. For more details, we refer to~\cite{Nualart2006}. Toward this end, let $T>0$ be given. We consider the space $\mathcal{H}=L^2\left([0,T]\times \mathcal{M}\right)$, where $\mathcal{M}$ is the counting measure on $\mathbb{N}$, namely, for $v\in\mathcal{H}$,
$$
v(t)=\sum_{k=1}^\infty v_k(t).
$$
We endow $\mathcal{H}$ with the inner product and the norm
\begin{equation}\label{innerprod-1}
\inner{u}{v}_{\mathcal{H}}:= \sum_{k=1}^\infty\int_0^T u_k(t) v_k(t) d t,\quad
\mbox{and}\quad
\norm{v}_{\mathcal{H}}:= \sqrt{\inner{v}{v}_{\mathcal{H}}}, \quad \ u,v\in \mathcal{H}.
\end{equation}
We fix an isonormal Gaussian process $W=\{W(h)\}_{h\in\mathcal{H}}$ on $\mathcal{H}$, defined on a suitable probability space $(\Omega, \mathscr{F},\mathbb{P})$, such that $\mathscr{F}=\sigma(W)$ is the $\sigma$-algebra generated by $W$.
Denote by $C_p^{\infty}(\mathbb{R}^n)$, the space of all smooth functions on $\R^n$ with at most polynomial growth partial derivatives.
Let $\mathcal{S}$ be the space of simple functionals of the form
\begin{equation*}
F = f(W(h_1), \dots, W(h_n)),\quad
f\in C_p^{\infty}(\mathbb{R}^n), \  h_i \in \mathcal{H},\ 1\leq i \leq n.
\end{equation*}
As usual, we define the Malliavin derivative $D$ on $\mathcal{S}$ by
\begin{equation}\label{def-MD}
DF=\sum_{i=1}^n  \frac {\partial f} {\partial x_i} (W(h_1), \dots, W(h_n)) h_i, \quad F\in \mathcal{S}.
\end{equation}
We note that the derivative operator $D$ is a closable operator from $L^p(\Omega)$ into $L^p(\Omega;  \mathcal{H})$,  for any $p \geq1$.
Let $\mathbb{D}^{1,p}$, $p \ge 1$, be the completion of $\mathcal{S}$
with respect to the norm
$$
\norm{F}_{1,p} = \left(\bE\big[ |F|^p \big] + \bE\big[  \norm{DF}^p_\mathcal{H} \big] \right)^{1/p}.
$$
Also, for $F$ of the form
$$
F=f\left(W\left(\mathbbm{1}_{[0,t_1]}\right),\dots,W\left(\mathbbm{1}_{[0,t_n]}\right)\right),
\quad t_1,\dots,t_n\in [0,T],
$$
we define the Malliavin derivative of $F$ at the point $t$ as
$$
D_t F=\sum_{i=1}^n \frac{\partial f}{\partial x_i}\left(W\left(\mathbbm{1}_{[0,t_1]}\right),\dots,W\left(\mathbbm{1}_{[0,t_n]}\right)\right)
\mathbbm{1}_{[0,t_i]}(t),\quad t\in [0,T],
$$
where $\mathbbm{1}_{A}$ denotes the indicator function of set $A$. For simplicity,
from now on, we define $W(t):=W\left(\mathbbm{1}_{[0,t]}\right)$, $t\in [0,T]$, to represent a standard Brownian motion on $[0,T]$.
If $\mathscr{F}$ is generated by  a collection of independent standard Brownian motions $\{W_k,\ k\geq 1\}$ on $[0,T]$, we define the Malliavin derivative of $F$ at the point $t$ by
\begin{equation}\label{def-MD-tj}
D_tF:=\sum_{k=1}^{\infty}D_{t,k}F:=\sum_{k=1}^{\infty}\sum_{i=1}^n \frac{\partial f}{\partial x_i}\left(W_k(t_1),\dots,W_k(t_n)\right)
\mathbbm{1}_{[0,t_i]}(t),\quad t\in [0,T].
\end{equation}
Next, we denote by $\boldsymbol{\delta}$, the adjoint of the Malliavin derivative $D$ (as defined in
\eqref{def-MD})
given by the duality formula
\begin{equation*}
\bE\left(\boldsymbol{\delta}(v) F\right) = \bE\left(\inner{v}{DF}_\mathcal{H}\right),
\end{equation*}
for $F \in \mathbb{D}^{1,2}$ and $v\in \mathcal{D}(\boldsymbol{\delta})$, where $\mathcal{D}(\boldsymbol{\delta})$ is the domain of $\boldsymbol{\delta}$. If $v\in  L^2(\Omega;\mathcal{H})\cap \mathcal{D}(\boldsymbol{\delta})$ is a square integrable process, then the adjoint $\boldsymbol{\delta}(v)$ is called the Skorokhod integral of the process $v$
(cf. \cite{Nualart2006}), and it can be written as
\begin{align*}
 \boldsymbol{\delta}(v)=\int_0^T v(t) dW(t).
\end{align*}

Now, the next theorem tell us how to calculate the Malliavin derivative for the Skorokhod integral of the process $v$ (in particular the derivative for a It\^{o} integral).

\begin{proposition}\cite[Theorem 1.3.8]{Nualart2006} \label{chain-rule-Malliavin}
Suppose that $v \in L^2(\Omega;\mathcal{H})$ is a square integrable process such that $v(t)\in \mathbb{D}^{1,2}$ for almost all $t\in [0,T]$.
Assume that the two parameter process $\{D_tv(s)\}$ is square integrable in $L^2\left([0,T]\times \Omega;\mathcal{H}\right)$.
Then, $\boldsymbol{\delta}(v) \in \mathbb{D}^{1,2}$ and
\begin{align}\label{chain-rule-formula}
D_t \left(\boldsymbol{\delta}(v)\right) = v(t) + \int_0^{T} D_tv(s) d W(s),\quad t\in [0,T].
\end{align}
\end{proposition}
Next, we present a connection between Malliavin calculus and Stein's method. For symmetric functions $f\in  L^2\big([0,T]^q\big)$, $q\geq 1$, let us define the following multiple integral of order $q$ 
\begin{equation*}
 \mathbb{I}_q(f)=q!\int_0^T d W(t_1)\int_0^{t_1} d W(t_2)\cdots \int_0^{t_{q-1}} d W(t_q) f(t_1,\ldots,t_q),
\end{equation*}
with $0<t_1<t_2<\cdots<t_q<T$. Note that $\mathbb{I}_q(f)$ is also called the $q$-th Wiener chaos \cite[Theorem 2.7.7]{NourdinPeccati2012}.
\begin{definition}
the total variation of two random variables $F$ and $G$ in $\mathbb{R}^d$, denote by $d_{TV}(F,G)$, is
$$
d_{TV}(F,G)=\sup_{B\in \mathbb{B}(\mathbb{R}^d)} \left|\mathbb{P}[F\in B]-\mathbb{P[G\in B} \right|,
$$
\end{definition} 

We present the next result about an upper bound for the total variation of a $q$-th multiple integral and a Gaussian random variable.
\begin{proposition} \cite[Theorem 5.2.6]{NourdinPeccati2012} \label{Th2.2}
Let $q\ge 2$ be an integer, and let $F=\mathbb{I}_q(f)$ be a multiple integral of order $q$ such that $\bE(F^2)=\sigma^2>0$. Then, for $\mathcal{N}=\mathcal{N}(0,\sigma^2)$,
\begin{equation*}
d_{TV}(F,\mathcal{N})\le \frac{2}{\sigma^2}\sqrt{\Var\left(\frac{1}{q}\norm{DF}_{\mathcal{H}}^2 \right)}.
\end{equation*}
\end{proposition}
With that proposition we can bound the total variation between $F$ and a normal distribution by the variance of the Malliavin derivative of $F$.  We present the main result that we will use to prove the asymptotic normality of the estimators. It tells us that the asymptotic normality can be equivalent to the asymptotic behavior of the total variation between the sequence $F_n$, that  we want to prove its asymptotic normality, and the normal distribution.

\begin{theorem} \cite[Corollary 5.2.8]{NourdinPeccati2012} \label{Equiv}
Let $F_N=\mathbb{I}_q(f_N)$, $N\ge 1$, be a sequence of random variables
for some fixed integer $q\ge 2$.
Assume that $\bE\left(F_N^2\right)\rightarrow \sigma^2>0$, as $N\rightarrow\infty$.
Then, as $N \rightarrow\infty$, the following assertions are equivalent:
\begin{enumerate}
\item $F_N\overset{d}\longrightarrow \mathcal{N}:=\mathcal{N}(0,\sigma^2)$;
\item $d_{TV}\left(F_N,\mathcal{N}\right)\longrightarrow 0$.
\end{enumerate}
\end{theorem}

\end{document}